\documentclass{article}
\usepackage[a4paper, total={6in, 8in}]{geometry}
\usepackage{multicol}
\usepackage{float}
\usepackage{caption}
\usepackage{subcaption}
 \makeatother
\usepackage{ntheorem}
\usepackage{caption}
\usepackage{dirtytalk}
\usepackage{outlines}
\usepackage{bm}
\usepackage{mathtools}
\usepackage{adjustbox}
\usepackage{color}
\usepackage{lipsum}
\usepackage{etoolbox}
\usepackage[usenames,dvipsnames]{pstricks}
 \usepackage{epsfig}
\usepackage{graphicx}
\usepackage{cuted}
\usepackage{xcolor}
\usepackage{verbatim}
\usepackage{marvosym}
\usepackage{array}
\usepackage{amssymb}
\usepackage{amsfonts}
\usepackage{enumerate}
\providecommand{\keywords}[1]{\textbf{\textit{Index terms---}} #1}
\newtheorem*{GOCP*}{GOCP}
\newtheorem{EOCP}{Equivalent OCP}
\newtheorem{OCP}{OCP}
\newtheorem*{solution*}{Solution}
\newtheorem{theorem}{Theorem}
\newtheorem{proposition}{Proposition}
\newtheorem{corollary}{Corollary}
\newtheorem{example}{Example}
\newtheorem*{proof}{Proof}
\newtheorem*{problem}{Problem}

\newtheorem*{remark1}{Remark}
\newtheorem*{remarks}{Remarks}
\newtheorem{lemma}{Lemma}
\begin{document}
\title{\textbf{ Some Insights on Synthesizing Optimal Linear Quadratic Controller Using Krotov's Sufficiency Conditions }}

\author{Avinash Kumar and Tushar Jain
	\thanks{Avinash Kumar and Tushar Jain are with Indian Institute of Technology Mandi, School of Computing and Electrical Engineering, Himachal Pradesh 175005, India. {email: \tt d16005@students.iitmandi.ac.in, tushar@iitmandi.ac.in}}
}
\date{}
\maketitle
\begin{abstract}

This paper revisits the problem of synthesizing the optimal control law for linear systems with a quadratic cost. For this problem, traditionally, the state feedback gain matrix of the optimal controller is computed by solving the Riccati equation, which is primarily obtained using Calculus of Variations (CoV) and Hamilton-Jacobi-Bellman (HJB) equation based approaches. To obtain the Riccati equation, these approaches requires some assumptions in the solution procedure, i.e. the former approach requires the notion of co-states and then their relationship with states is exploited to obtain the closed form expression for optimal control law, while the latter requires an \textit{a-priori} knowledge regarding the optimal cost function. In this paper, we propose a novel method for computing linear quadratic optimal
control laws by using the global optimal control framework introduced by V.F. Krotov. As shall be illustrated in this article, this framework does not require the notion of co-states and any \textit{a-prior} information regarding the optimal cost function. Nevertheless, using this framework, the optimal control problem gets translated to a non-convex optimization problem. The novelty of the proposed method lies in transforming the non-convex optimization problem into a convex problem. The insights along with the future directions of the work are presented and gathered at appropriate locations in the article. Finally, numerical results are provided to demonstrate the proposed methodology.
\end{abstract}
\keywords{Optimal control, sufficient optimality conditions, global optimality, linear systems, Ricatti equations, Krotov function}

\section{Introduction}
Optimal control theory is a heavily explored and still developing field of control engineering where  the objective is to design a control law so as to optimize (maximize or minimize) performance index (cost functional) while driving the states of a dynamical system to zero (Regulation problem) or to make output  track a reference trajectory (Tracking problem) \cite{Des02}. The generic optimal control problem (GOCP) is given as:
\vspace{0.2 cm}

({\textit{Notation: Throughout this article the small alphabets represent scalar quantities, small bold alphabets represent vector quantities  and the capital alphabets represent matrices.}})
\begin{GOCP*} \label{gocp}
Compute an optimal control law $\boldsymbol{u}^*(t)$ which minimizes (or maximizes) the performance index/cost functional:
\begin{equation}  \label{g_p_i}
J(\boldsymbol{x}(t),\boldsymbol{u}(t),t)= l_f(\boldsymbol{x}(t_f)) + \int_{t_0}^{ t_f} l(\boldsymbol{x}(t),\boldsymbol{u}(t),t) dt 
\end{equation}
subject to the system dynamics
$ \dot{\boldsymbol{x}}(t) = f(\boldsymbol{x}(t),\boldsymbol{u}(t),t) \ \text{with } {\boldsymbol{x}}(t_0) \in \mathbb{R}^n; \ t \in  [t_0,t_f]  $
to give the desired optimal trajectory $\boldsymbol{x}^*(t)$. Here, $l(\boldsymbol{x}(t),\boldsymbol{u}(t),t)$ is the running cost, $l_f(\boldsymbol{x}(t_f))$ is the terminal cost, $\boldsymbol{x}(t) \in  \mathbb{X} \subset \mathbb{R}^n$ is the state vector and $\boldsymbol{u}(t) \in \mathbb{U} \subset \mathbb{R}^m$ is the control input vector to be designed. Also, $l_f : \mathbb{R}^n \rightarrow \mathbb{R} $ and $l : \mathbb{R}^n \times \mathbb{R}^m \times [t_0, t_f] \rightarrow \mathbb{R}$ are continuous. 
\end{GOCP*}

Since, the aforementioned problem corresponds to optimization of the cost functional subject to dynamics of the system considered and possibly constraints on input(s) and/or state(s) the Calculus of Variations(CoV) is generally employed  to address optimal control design problems \cite{Des02,Kro95}. The assumption of an optimal control is usually the first step while using CoV techniques. Subsequently, the conditions which must be satisfied by such an optimal control law are derived. Hence, only  necessary conditions are found and sufficiency of these conditions is not guaranteed. Furthermore,  the obtained control law is usually only locally  optimum. Nevertheless, there are results available in the literature which provide restrictions under which the necessary conditions indeed  become sufficient and the global optimal control law is obtained \cite{Olvi66,Kva13,Kam71}. Note that, in solving optimal control design problems, the CoV method uses the notion of so-called co-states (which are not actually present in the system). Moreover, in the solution procedure, the existence of a linear relationship between the states and co-states is exploited to compute the closed form of optimal control law (this is particularly true for linear quadratic problems). See \cite{kalman1960contributions} for more details.

Alongside CoV, another tool, namely dynamic programming (DP) (introduced by Bellman), has also been explored to solve optimal control problems. The application of DP to optimal control design problems for continuous linear systems leads to the celebrated Hamilton-Jacobi-Bellman (HJB) equation which also gives a necessary condition for optimality \cite{Des02}. Nevertheless, this equation also provides sufficiency under the following mild conditions on the optimal cost function \cite{Kro95,athans2013optimal}:
$(a)${ there exists an continuously differentiable optimal cost function} and
$(b)${ the gradient of cost function with respect to state vector equals co-state which corresponds to the optimal trajectory.}
For example, consider the optimal control design problem for the system \cite{beard1998approximate}-
$ 	\dot{x}(t)=x(t)u(t)  $
with performance measure as
$ 	J\big(x(t),u(t),t\big)= \int_{1}^{\infty}\big[ x^2(t)+u^2(t)\big] dt.  $
For this problem, the optimal cost function  is $J^\ast(x^\ast(t),u^\ast(t),t) = |x^\ast(t)| $ and hence the HJB equation is not defined at $x=0$ because of non-differentiability of $J^\ast(x^\ast(t),u^\ast(t),t)$. From the aforementioned observations, a solution method for optimal control design which does not require these conditions is desirable. As shall be demonstrated in this article,  Krotov solution methodology is indeed such a methodology. In fact, this methodology  provides sufficient conditions for the existence of the global optimal control law without using the notion of co-states and any \textit{a-priori} information regarding the optimal cost function \cite{Kro95}.

Starting in the sixties, the results on sufficient conditions  for the global  optimum of optimal control problem were published by Vadim Krotov \cite{krotov1962methods,Kro88}. The conditions have been derived from the so-called of extension principle \cite{gurman2016certain}. The first step, while employing these conditions, is a total decomposition of the OCP with respect to time via an appropriate selection of the so-called Krotov function  \cite{Kro95,HalAgr17,Raf18}. Once such a decomposition is obtained, the problem is reduced to a family of independent elementary optimization problems parameterized in time $t$. It has been shown in \cite{Kro95} that the two problems- original OCP and the  optimization problem resulting from decomposition- are completely equivalent. The method, however, is abstract in the sense that the selection of Krotov function is not straightforward and the selection is very problem specific \cite{gurman2016certain}. A number of works have used Krotov methodology for solving OCPs encountered in control of structural vibration problems in buildings \cite{HalAgr17}, MEMS-based energy harvesting problem \cite{Raf18}, magnetic resonance systems \cite{vinding2012}, quantum control \cite{schirmer2011}, computation of extremal space trajectories \cite{azimov2017analytical,krotov2005national} etc. However, the equivalent optimization problems in all these articles are non-convex and hence the iterative methods, one of them being Krotov method, are employed to obtain their solutions. To address this issue, we propose a novel method to directly (non-iteratively) synthesize optimal controllers for linear systems using Krotov sufficient conditions. The innovation in our approach lies in transforming the non-convex optimization problem into a convex optimization problem by a proper selection of Krotov functions.

Some related (preliminary) results of the approach were reported in \cite{Avi19icc} for linear quadratic regulation problem. In \cite{Avi19acc}, the methodology was demonstrated for finite horizon linear quadratic  optimal control problems and the Krotov function was taken to be a \textit{positive definite quadratic function}. This assumption was relaxed in \cite{Avi19ccta} and preliminary results were reported. This article is presents a rather detailed discussion of the results along with extension of the methodology to infinite horizon problems. It differs from \cite{Avi19acc} in that the Krotov function here is taken to be a quadratic function, neither symmetry nor positive definiteness is imposed upon this function.

Considering the aforementioned points and given the fact that Krotov framework remains highly unexplored in literature (to the best of author's knowledge), this article may serve as a background for further exploration of this framework to more involved control problems  viz. nonlinear optimal control, distributed optimal control etc. In summary, the contribution of this work is as follows: $(i)$ It exhaustively describes the methodology for solving the standard linear quadratic optimal control problems (both finite and infinite horizon) using Krotov sufficient conditions, $(ii)$ It solves the equivalent optimization problems via convexity imposition: a technique which is not used in the previous work which use Krotov conditions and then the analysis of resulting LMI is  presented and $(iii)$ It provides the insights which result upon synthesizing the optimal control laws and may also lay the foundation upon which the Krotov sufficient conditions may be employed for solving more complex optimal control problems viz. nonlinear optimal control problems. Note that the a preliminary work in this direction, specifically, for solving scalar nonlinear optimal control problems using Krotov conditions was reported in \cite{Avi19ecc}.

The rest of the article is organized as follows. In Section \ref{pf}, the preliminaries of linear quadratic optimal control problems are discussed and the solution methodologies based on the Calculus of Variations (CoV) based method and Hamilton-Jacobi-Bellman (HJB) equation based method are outlined. The  assumptions as encountered in these approaches are also discussed. In Section \ref{main_res}, the background literature of Krotov sufficient conditions and their application to the problems considered is detailed. This section also discusses a number of insights which result while solving the considered optimal control problems. These insights are gathered as remarks at appropriate locations in this section. The Krotov iterative method is also discussed in brief in this section. In Section \ref{ne}, the proposed method is demonstrated through numerical examples. Finally, the concluding remarks and future scope of the work presented in Section \ref{conc}.

\section{Preliminaries and Problem Formulation} \label{pf}
In this section, solution procedures of Linear Quadratic Regulation (LQR) and Linear Quadratic Tracking (LQT) problems using CoV and HJB equation based approaches are briefly discussed in order to concretely highlight the assumptions used in these approaches. 
\begin{OCP} \label{prob_lqr}
(Finite Horizon LQR problem):\\ 
{Compute an optimal control law $\boldsymbol{u}^*(t)$ which minimizes the quadratic performance index/cost functional:
\begin{align*}
J(\boldsymbol{x}(t),\boldsymbol{u}(t),t)= 0.5\Big[ \boldsymbol{x}^T(t_f) F(t_f) \boldsymbol{x}(t_f)\Big] 
+ 0.5 \Big[ \int_{t_0}^{ t_f} \boldsymbol{x}^T(t)Q(t){\boldsymbol{x}}(t)+\boldsymbol{u}^T(t)R(t)\boldsymbol{u}(t) dt \Big]
\end{align*}
subject to the system dynamics
$ \dot{\boldsymbol{x}}(t)=A(t)\boldsymbol{x}(t)+B(t)\boldsymbol{u}(t)  $
and drives the states of system to zero (Regulation).
Here $ \boldsymbol{x}(t_0)=\boldsymbol{x}_0$ is given, $\boldsymbol{x}(t_f)$ is free  and  $t_f$ is fixed.
Also, $Q(t) \succeq0 \text{ and } R(t)  \succ0 \  \forall \ t \in [t_0,t_f].$
}
\end{OCP}
The solution using CoV technique comprises of four major steps:
\begin{enumerate}[i)]
\item{\textit{Formulation of Hamiltonian function:} The Hamiltonian for the considered problem is given as:
\begin{align*}  \label{hamil}
\mathcal{H}(\boldsymbol{x},\boldsymbol{u},{\boldsymbol{\lambda}})=  \frac{1}{2} \boldsymbol{x}^{T}(t) Q(t) \boldsymbol{x}(t)+ \frac{1}{2} \boldsymbol{u}^{T}(t) R(t) \boldsymbol{u}(t)
+ \boldsymbol{\lambda}^{T}(t) [A(t) \boldsymbol{x}(t)+B(t)\boldsymbol{u}(t)] 
\end{align*}
where $\boldsymbol{\lambda}(t)$ is the co-state vector.
}
\item{\textit{Obtaining Optimal Control law using first order necessary condition:} The optimal control law $\boldsymbol{u}^{*}(t)$ is obtained as:
\begin{align*}
\frac{\partial \mathcal{H}}{\partial \boldsymbol{u}}= 0 
\implies \boldsymbol{u}^\ast(t)= - R^{-1}(t) B^{T}(t) \boldsymbol{\lambda}^\ast(t)
\end{align*}
}
\item{\textit{Use of State and Co-state Dynamics and a transformation to connect state and co-state for all $t \in [t_0,t_f]$:} The boundary conditions (i.e $t_f$ being fixed and $\boldsymbol{x}(t_f)$ being free) lead to the following boundary condition on $\boldsymbol{\lambda}(t)$:
$
   \boldsymbol{\lambda}^\ast(t_f)=F(t_f) \boldsymbol{x}^\ast(t_f)
 $. Then, the following transformation to connect co-state and state is used:
\begin{equation} 
\label{lambda_all}   
\boldsymbol{{\lambda}}^\ast(t)= P(t) \boldsymbol{x}^\ast(t)
  \end{equation}
to compute closed form the optimal control law as:
\begin{equation*}  \boldsymbol{u}^\ast(t)= - R^{-1}(t) B^{T}(t) P(t) \boldsymbol{x}^\ast(t) \end{equation*}}
\item{ \textit{Obtaining Matrix Differential Riccati Equation:} Finally, taking the derivative of equation \eqref{lambda_all} and substituting the state and co-state relations the following matrix differential Riccati equation (MDRE) is obtained which $P(t)$ must satisfy for all $t \in [t_0,t_f]$:
\begin{align*}  \dot{P}(t)+P(t)A(t)+A^{T}(t)P(t)+Q(t) - P(t) B(t) R^{-1}(t) B^{T}(t)P(t)=0   \end{align*}

}

\end{enumerate}
Further, the solution using HJB equation requires 
\begin{equation} \label{hjb_eqn}
 \frac{ \partial J^{\ast}} {\partial t}   + \mathcal{H} \left[  \boldsymbol{x}^\ast (t), \frac{\partial J^*}{\partial \boldsymbol{x}^*},\boldsymbol{u}^*(t), t \right]=0 \ 	
  \ \forall t \in  \ [t_0,t_f]
\end{equation} 
where $J^\ast(\boldsymbol{x}^\ast(t),t)$ is the optimal cost function, 
and $\boldsymbol{u}^*(t)$ is the optimal control law. To solve \eqref{hjb_eqn}
the boundary condition is given as :
$
\label{J_b_c}    J^\ast (\boldsymbol{x}^\ast(t_f),t_f) = \frac{1}{2} \boldsymbol{x}^{\ast T}(t_f) F(t_f) \boldsymbol{x}^\ast(t_f) $
with $J^\ast(\boldsymbol{x}^{\ast}(t),t)$ \textit{assumed} to be
\begin{equation} 
\label{eq_P}
J^\ast(\boldsymbol{x}(t),t)=\frac{1}{2} \boldsymbol{x}^{\ast T}(t) P(t)\boldsymbol{x}^\ast(t)
\end{equation}
where $P(t)$ is a real, symmetric, positive-definite matrix to be determined.
Substituting \eqref{eq_P} into \eqref{hjb_eqn}, we get:
\begin{multline*}
\frac{1}{2} \boldsymbol{x}^{\ast T}(t) \dot{P}(t)\boldsymbol{x}^\ast (t)+\frac{1}{2}\boldsymbol{x}^{\ast T}(t)P(t)A(t)\boldsymbol{x}^\ast(t)+ \frac{1}{2} \boldsymbol{x}^{\ast T}(t) Q(t)\boldsymbol{x}^\ast (t) +\frac{1}{2} \boldsymbol{x}^{\ast T}(t)A^{T}(t)P(t)\boldsymbol{x}^\ast(t) \\
-\frac{1}{2} \boldsymbol{x}^{\ast T}(t) P(t) B(t) R^{-1}(t)B^{T}(t)P(t) \boldsymbol{x}^\ast (t) =0
\end{multline*}
This equation is valid for {any} $\boldsymbol{x}^\ast(t)$, if:
\begin{align*}
 \dot{P}(t) +Q(t)+P(t)A(t)+A^{T}(t)P(t)- P(t)B(t) R^{-1}(t)B^{T}(t) P(t) =0
\end{align*}
Finally, $P(t_f)=F(t_f)$ and thus the solution is same as that obtained using CoV. Summarizing above, 
the global optimal control law is given by 
\begin{equation*}
\boldsymbol{u}^\ast(t)= - R^{-1}(t) B^{T}(t) P(t) \boldsymbol{x}(t)
\end{equation*}
 where $P(t)$ is the solution of
 \begin{align*}
 \dot{P}(t)+P(t)A(t)+A^{T}(t)P(t) +Q(t) - P(t)B(t)R^{-1}(t)B^{T}(t)P(t)=0
  \end{align*}
   with boundary condition $P(t_f)=F(t_f)$.
\begin{OCP} \label{prob_lqt}
(Finite Horizon LQT problem):\\ 
{Compute an optimal control law $\boldsymbol{u}^*(t)$ which minimizes the quadratic performance index/cost functional:
\begin{align*} J=  0.5\Big[ \boldsymbol{e}^T(t_f) F(t_f) \boldsymbol{e}(t_f)\Big]+ 0.5 \Big[ \int_{t_0}^{ t_f} \boldsymbol{e}^T(t)Q(t)\boldsymbol{e}(t)+\boldsymbol{u}^T(t)R(t)\boldsymbol{u}(t) dt \Big]
\end{align*}
where 
$\boldsymbol{e}(t) \triangleq \boldsymbol{z}(t)-\boldsymbol{y}(t)$
subject to the system dynamics
  \begin{align*}
\dot{\boldsymbol{x}}(t)&=A(t)\boldsymbol{x}(t)+B(t)\boldsymbol{u}(t) \\
\boldsymbol{y}(t)&=C(t)\boldsymbol{x}(t)
\end{align*}
such that the output $\boldsymbol{y}(t)$ tracks the desired reference trajectory $\boldsymbol{z}(t)$.
Here
 $\boldsymbol{e}(t) \triangleq \boldsymbol{y}(t)-\boldsymbol{z}(t)$ is the error vector, $\boldsymbol{x}(t_0)=\boldsymbol{x}_0$ is given, $\boldsymbol{x}(t_f)$ is free and $t_f$ is fixed.
Also, $Q(t) \succeq0 \text{ and } R(t) \succ  0 \ \forall \ t \in [t_0,t_f].$
}
\end{OCP}
Similarly to the solution of LQR, the CoV and HJB equation based approaches yield the optimal control law as:
\begin{align*}
&\boldsymbol{u}^*(t)= - R^{-1}(t) B^T(t) P(t) \boldsymbol{x}^*(t) + R^{-1}(t) B^T(t)\boldsymbol{g}(t) \nonumber 
\end{align*}
where ${P}(t)$ and  $\boldsymbol{g}(t)$ satisfy:
\begin{align*}
 \dot{P}(t)+P(t)A(t)+A^T(t)P(t) + C^T(t)Q(t)C(t)
-P(t)B(t)R^{-1}(t)B^T(t)P(t)=0 \ \forall t \in [t_0,t_f]
\end{align*}
and
\begin{align*}
\dot{\boldsymbol{g}}(t)+\big[ A(t) - B(t)R^{-1}(t)B^T(t)P(t) \big]^{T}\boldsymbol{g}(t)
+C^T(t)Q(t)\boldsymbol{z}(t)= \boldsymbol{0} \ \forall t \in [t_0,t_f]
\end{align*}
with
$P(t_f)=C^T(t_f)F(t_f)C(t_f)$ 
and
$\boldsymbol{g}(t_f)=C^T(t_f)F(t_f)C(t_f) $
respectively.
Note, that in HJB approach  the optimal cost function has to be guessed.

Although the CoV and HJB based approaches as described above are widely employed for solving OCPs, there are some assumptions associated with these approaches in their solution procedure. Specifically, the CoV based approach uses the notion of and co-states and their relationship with states for all time \eqref{lambda_all} to compute the optimal control law. Similarly, the HJB based approach requires the existence of the continuously differentiable  optimal cost function and that its gradient with respect to state is the co-state corresponding to the optimal trajectory \cite{athans2013optimal}.  Thus, the information about the optimal cost function must be known \textit{a priori}. The angle of our attack is to synthesize an optimal control law using Krotov sufficient conditions, where the above issues are not encountered in the solution procedure. However, it is well known that the control law using these conditions is synthesized through an iterative procedure. The main non-trivial issue to be tackled is to obtain non-iterative solutions of optimal control problems using Krotov conditions. The next section answers this question for linear quadratic optimal control problems.

\section{Computation of Optimal Control Laws } \label{main_res}
 In this section, solutions of the LQR and LQT problems using Krotov sufficient conditions are detailed.

 \subsection{Krotov  Sufficient Conditions in Optimal Control}
 The underlying idea behind Krotov sufficient conditions for global optimality of control processes is the total decomposition of the original OCP with respect to time using the so-called {extension principle} \cite{Kro95}.

 \subsubsection{Extension Principle}

 The essence of the extension principle is to replace the original optimization problem with complex relations and/or constraints  with a simpler one such that they are excluded in the new problem definition but the solution of the new problem still satisfies the discarded relations \cite{gurman2016certain}.

 Consider  a scalar valued functional $I(\boldsymbol{v})$ defined over a set $\mathbb{M}$ (i.e. $\boldsymbol{v} \in \mathbb{M}$) and the optimization problem  as\\
\textit{Problem $(i)$:}
 Find $\bar{\boldsymbol{v}}$ such that $d=\inf_{\boldsymbol{v} \in \mathbb{M}}I(\boldsymbol{v})$ where $d \triangleq I(\bar{\boldsymbol{v}}) $. 
 
 \noindent Instead of solving the Problem $(i)$, another equivalent optimization problem is solved.  
 Let $L$ denote the equivalent representation of the original cost functional.
 Then, a new problem is formulated over $\mathbb{N}$, a super-set of $\mathbb{M}$ as:
 
 \noindent \textit{Equivalent Problem $(i)$:} 
 Find $\bar{\boldsymbol{v}}$ such that $e= \inf_{\boldsymbol{v }\in \mathbb{N}}L(\boldsymbol{v})$ where $e \triangleq L(\bar{\boldsymbol{v}}) $. 
 
 The equivalent problem is also called the extension of the original problem. The key idea is that solving the equivalent problem may be much simpler than solving the original problem. The method of choosing of the \textit{equivalent functional} $L$ is not unique and the selection is generally made according to specifications of the problem under consideration. This \textit{freedom} in the selection of the equivalent functional can be exploited to tackle the generic non-convex optimization problems. Also, it is necessary to ensure that:
 $
 I(\boldsymbol{v})=L(\boldsymbol{v}) \  \forall \ \boldsymbol{v}\in \mathbb{M}
 $
 so that the optimizer $\bar{v}$ is actually the optimizer of the original Problem $(i)$ \cite{Kro95}. 
  Clearly, the application of extension principle requires \textit{appropriate} selection of the equivalent functional $L$ and the set $\mathbb{N}$.

 \subsubsection{Application of Extension Principle to Optimal Control Problems}

The equivalent problem of the GOCP which is obtained by the application of the extension principle is given in the following theorem which provides a concrete definition of the equivalent functional $L$ and the set $\mathbb{N}$  for this problem.
\begin{theorem}\label{thm1}
For the GOCP, let $q(\boldsymbol{x}(t),t)$ be a continuously differentiable function. Then, there is an equivalent representation of \eqref{g_p_i} given as:
\begin{equation*} 
J_{eq}(\boldsymbol{x}(t),\boldsymbol{u}(t))= s_f(\boldsymbol{x}(t_f)) +q(\boldsymbol{x}(t_0),t_0)+ \int_{t_0}^{ t_f} s(\boldsymbol{x}(t),\boldsymbol{u}(t),t) dt 
\end{equation*}
where
\begin{align*}
s(\boldsymbol{x}(t),\boldsymbol{u}(t),t)  & \triangleq  \frac{\partial q}{\partial t}+ \frac{\partial q}{\partial \boldsymbol{x}} {f} (\boldsymbol{x}(t), \boldsymbol{u}(t),t)+ l(\boldsymbol{x}(t),\boldsymbol{u} (t),t) \\
s_f(\boldsymbol{x}(t_f) &\triangleq l_f(\boldsymbol{x}(t_f)) -q(\boldsymbol{x}(t_f),t_f)
\end{align*}
\end{theorem}
\begin{proof}
See \cite[Section $2.3$]{Kro95} for the proof.
\end{proof}
The equivalent functional $J_{eq}(\boldsymbol{x}(t),\boldsymbol{u}(t))$ leads to a sufficient condition for the global optimality of an admissible process  $\text {i.e. a } \left [ \boldsymbol{x}(t),\boldsymbol{u}(t) \right ] \text{ pair which satisfies the dynamical equation } \dot{\boldsymbol{x}}(t)= f(\boldsymbol{x}(t),\boldsymbol{u}(t),t)$ and the input/state constraints.

\begin{theorem}\textit{(Krotov Sufficient Conditions)} \label{thm2} 
If $\big({\boldsymbol{x}^*(t)},{\boldsymbol{u}^*(t)}\big)$ is an \textit{admissible process}  such that
\begin{align*}
s(\boldsymbol{x}^*(t),\boldsymbol{u}^*(t),t)=   \min_{\boldsymbol{x} \in \mathbb{X}, \boldsymbol{u} \in \mathbb{U} }  s(\boldsymbol{x}(t),\boldsymbol{u}(t),t),  \forall t  \in   [t_0,t_f)
\end{align*}
and
\begin{equation*}
s_f(\boldsymbol{x}^*(t_f)) = \min_{\boldsymbol{x} \in \mathbb{X}_f}  s_f(\boldsymbol{x})
\end{equation*}
then $(\boldsymbol{x}^*(t), \boldsymbol{u}^*(t))$ is an optimal process.
Here, $\mathbb{X}_f$ is the terminal set for admissible $\boldsymbol{x}(t)$ of $ \dot{\boldsymbol{x}}(t) = f(\boldsymbol{x}(t),\boldsymbol{u}(t),t) $ i.e. if $\boldsymbol{x}(t)$ is admissible then $\boldsymbol{x}(t_f) \in \mathbb{X}_f $.
\end{theorem}

\begin{proof}
{\it Proof} See \cite[Section $2.3$]{Kro95} for the proof.
\end{proof}
\begin{remarks}
\normalfont
Some remarks which follow from Theorem \ref{thm1} and Theorem \ref{thm2} are now briefed.
\begin{enumerate}
\item {The functional $J_{eq}$ is the equivalent functional of the original functional $J$ in \eqref{g_p_i}. Specifically, $\left [ J_{eq}= J \ \forall \left[ \boldsymbol{x}(t), \boldsymbol{u}(t) \right] : \dot{\boldsymbol{x}} (t) =f(\boldsymbol{x}(t),\boldsymbol{u}(t),t) \right]$.}
\item {The function $q(\boldsymbol{x}(t),t)$, known as \textit{Krotov function}, can be \textit{any} continuously differentiable function and each $q(\boldsymbol{x}(t),t)$ leads to different equivalent functional $J_{eq}$. A general way to choose this function is not known and the selection is usually done according to the specifics of the problem in hand. Moreover, it is possible  that the same solution results for different selections of $q (\boldsymbol{x}(t))$ \cite{Kro95}. While this \textit{ad-hocness} in choosing the Krotov function may seem burdensome, it provides enough freedom in making the selection according the specifications of the problem in hand \cite{gurman2016certain,salmin17}.} In essence, this work \textit{exploits} the \textit{ad-hocness} in selecting the Krotov function.
\item The original GOCP is transformed to an equivalent optimization problem of the functions $s$ and $s_f$ over the sets $\mathbb{X}, \mathbb{U}$ and $\mathbb{X}_f$. Thus, the dynamical equation constraint: $ \dot{\boldsymbol{x}} =f(\boldsymbol{x}(t),\boldsymbol{u}(t),t))$ is excluded from the optimization problem by defining $J_{eq}(\boldsymbol{x}(t),\boldsymbol{u}(t),t)$.
\item If the optimization problem formulated in Theorem \ref{thm2} is feasible (i.e. an  admissible process satisfying the sufficient conditions in Theorem \ref{thm2} can be found) then the function $q(\boldsymbol{x}(t),t)$ is called as the \textit{solving function}.
\item Clearly, different different selection of $q(\boldsymbol{x}(t),t)$ result in different optimization problems in Theorem \ref{thm2} and thus the selection of $q(\boldsymbol{x}(t),t)$ is crucial to effectively solve the optimization problems in Theorem \ref{thm2}. Furthermore, the necessary conditions for existence the optima (optimum) of the functions $s$  and $s_f$ coincide with the popular Pontryagin's minimum principle \cite{salmin17}. Moreover, a specific setting of Krotov function $q(\boldsymbol{x}(t),t)$ leads to the popular Hamilton-Jacobi-Bellman equation \cite{Kro95,salmin17}. These two observations lead to the conclusion that Krotov sufficient conditions are in fact the most general sufficient conditions for global results in optimal control theory.
\end{enumerate}
\end{remarks}

The solution to the equivalent optimization problem given in Theorem \ref{thm2} is generally computed using sequential methods because the resulting optimization problem is a non-convex. Such a sequence of processes is called an \textit{optimizing sequence}\cite{Kro88}. One such iterative method is the Krotov method in which an \textit{improving function} is chosen at each iteration. To avoid iterative methods, we convexify the equivalent problems in Theorem \ref{thm2} via a suitable selection of Krotov function.
In the next subsection, the suitable Krotov functions for LQR and LQT problems are proposed. In the following, we use $2J(\boldsymbol{x}(t),\boldsymbol{u}(t),t)$ as the cost functional instead of $J(\boldsymbol{x}(t),\boldsymbol{u}(t),t)$ and the time variable $t$ is dropped wherever it is required for the sake of simplicity.

\subsection{Solution of OCP \ref{prob_lqr} (LQR problem)} \label{sol_lqr}

The equivalent optimization problem for OCP \ref{prob_lqr}, as per Theorem \ref{thm2}, is given as:
\begin{EOCP} \label{eocp_1} Compute an admissible pair $ \left( \boldsymbol{x}^\ast, \boldsymbol{u}^\ast \right)$ which
\begin{enumerate} [(i)]
\item {$\min\limits_{(\boldsymbol{x},\boldsymbol{u}) \in \mathbb{R}^n \times \mathbb{R}^m}  s({\boldsymbol{x}},{\boldsymbol{u}},t), \forall t \in \  [t_0,t_f)$, where
 \begin{equation*}
  s=\frac{\partial q}{\partial t}+\frac{\partial q}{\partial \boldsymbol{x}}[A\boldsymbol{x}+B\boldsymbol{u}]+\boldsymbol{x}^TQ\boldsymbol{x}+\boldsymbol{u}^TR\boldsymbol{u} ;
  \end{equation*} 
}
\item $\min\limits_{\boldsymbol{x }\in \mathbb{R}}  s_f(\boldsymbol{x}) $, where 
\begin{equation*}
  s_f= \boldsymbol{x}^T(t_f)F(t_f)\boldsymbol{x}(t_f) - q(\boldsymbol{x}(t_f),t_f). 
  \end{equation*}
\end{enumerate} 
\end{EOCP}

The next proposition is one of the main results of the paper, where we propose a suitable Krotov function which will be useful in computing a direct solution to OCP $1$.

\begin{proposition}  \label{prop_1}
For the Equivalent OCP \ref{eocp_1}, let the Krotov function be chosen as 
\begin{equation} \label{q_1}
q(\boldsymbol{x},t)= \boldsymbol{x}^TP\boldsymbol{x}
\end{equation}
and every element of the matrix $P$ is differentiable $\forall \ t \in [t_0,t_f)$.
Then, the following statements  are equivalent:
\begin{enumerate} [ (a)]
\item{$s(\boldsymbol{x},\boldsymbol{u},t)$  and $s_f(\boldsymbol{x}(t_f))$ are convex functions in $(\boldsymbol{x}, \boldsymbol{u})$ and $\boldsymbol{x}(t_f)$ respectively;}
\item{$P$ satisfies the matrix inequalities:
\begin{enumerate}[(i)]
    \item {$\dot{P}+PA+A^TP+Q-\frac{1}{2} PBR^{-1}B^{T}P -\frac{1}{4} PBR^{-1}B^{T}P^T -\frac{1}{4}  P^TBR^{-1}B^{T}P  \succeq 0 $ $ \forall  t\in [t_0,t_f)$
    }
    \item{$F(t_f)-P(t_f) \succeq 0$}
\end{enumerate}
}
\end{enumerate}
\end{proposition}

\begin{proof}  
\normalfont
 With $q$ selected as in \eqref{q_1}, the function $s(\boldsymbol{x},\boldsymbol{u},t)$ becomes: 
\begin{align*}
s= & \boldsymbol{x}^{T} \dot{P} \boldsymbol{x}+ \boldsymbol{x}^T (P+P^{T})  [A\boldsymbol{x}+B\boldsymbol{u}]+\boldsymbol{x}^TQ\boldsymbol{x}+\boldsymbol{u}^TR\boldsymbol{u }\\
=& \boldsymbol{x}^T \dot{P} \boldsymbol{x}+\boldsymbol{x}^TPA\boldsymbol{x}+\boldsymbol{x}^TP^TA\boldsymbol{x}+\boldsymbol{x}^TP^TB\boldsymbol{u}+\boldsymbol{x}^TPB\boldsymbol{u}+\boldsymbol{x}^TQ\boldsymbol{x}+\boldsymbol{u}^TR\boldsymbol{u} 
\end{align*}
Adding and subtracting the term $\boldsymbol{x}^T \left( \frac{1}{2} PBR^{-1}B^{T}P +\frac{1}{4} PBR^{-1}B^{T}P^T + \frac{1}{4}  P^TBR^{-1}B^{T}P \right) \boldsymbol{x}$ , we get
\begin{align}\label{s1}
s= \boldsymbol{x}^T&( \dot{P}+PA+A^TP+Q-\frac{1}{2} PBR^{-1}B^{T}P -\frac{1}{4} PBR^{-1}B^{T}P -\frac{1}{4}  PBR^{-1}B^{T}P)\boldsymbol{x} \nonumber \\
&+\boldsymbol{x}^TPB \boldsymbol{u}+\boldsymbol{x}^TP^TB^T\boldsymbol{x}+\boldsymbol{u}^TR\boldsymbol{u} + \frac{1}{2} \boldsymbol{x}^TPBR^{-1}B^{T}P\boldsymbol{x}  \nonumber\\
& \hspace{0.3 cm}+\frac{1}{4} \boldsymbol{x}^T PBR^{-1}B^{T}P^T \boldsymbol{x} +\frac{1}{4} \boldsymbol{x}^T  P^TBR^{-1}B^{T}P \boldsymbol{x}
\end{align}
Since $R \succ {0}$, there exists a {unique positive definite matrix} \cite{SteLie}, say $\tilde{R}$, such that
$ \tilde{R}^2=R$ and  $\tilde{R}^{-1} \tilde{R}^{-1}=R^{-1} $.  
Now rearranging the terms in \eqref{s1}, we get
\begin{align}
s=&\boldsymbol{x}^T \Big[ \dot{P}+PA+A^TP-\frac{1}{2} PBR^{-1}B^{T}P -\frac{1}{4} PBR^{-1}B^{T}P^T - \frac{1}{4}  P^TBR^{-1}B^{T}P \Big]\boldsymbol{x} \nonumber \\
& \hspace{0.5 cm}+ \left[\tilde{R} \boldsymbol{u}+   \frac{1}{2}\tilde{R}^{-1} \left( B^T P^T +B^TP \right)\boldsymbol{x}\right]^T  \left[ \tilde{R} \boldsymbol{u}+  \frac{1}{2}\tilde{R}^{-1} \left( B^T P^T +B^TP \right)\boldsymbol{x} \right]  \label{seqn}
\end{align}
Clearly the second term in \eqref{seqn} is strictly convex. Now, $s$ is convex iff the following condition is satisfied
\begin{align*} 
 \dot{P}+PA+A^TP-\frac{1}{2} PBR^{-1}B^{T}P -\frac{1}{4} PBR^{-1}B^{T}P^T - \frac{1}{4}  P^TBR^{-1}B^{T}P \ \succeq 0, \forall  t  \in [t_o,t_f)
\end{align*}
Moreover, with $q$ as in \eqref{q_1}, $s_f(x(t_f))$ is given as 
\begin{align*}
s_f&= \boldsymbol{x}^T(t_f)F(t_f)\boldsymbol{x}(t_f)- \boldsymbol{x}^T(t_f) P(t_f)\boldsymbol{x}(t_f) \nonumber \\
&=\boldsymbol{x}^T(t_f) \big[F(t_f) - P(t_f) \big] \boldsymbol{x}(t_f) \label{s_eqn_f_lqr}
\end{align*}
Finally, $s_f$ is convex iff
\begin{equation*} \label{p_con_2}
F(t_f)-P(t_f) \succeq 0.
\end{equation*}
\end{proof}

\begin{corollary} \label{cor_1}
With $q(\boldsymbol{x},t)$ selected as in Proposition \ref{prop_1}, the following statements are true:
\begin{enumerate}
\item The function $q=\boldsymbol{x}^TP\boldsymbol{x}$ is a solving function for OCP \ref{eocp_1}.
\item The global optimal control law for OCP \ref{prob_lqr} is given by:
\begin{equation*} \label{u_eqn}
\boldsymbol{u}^*=- \frac{1}{2} {R}^{-1} B^T\left( P^T +P \right) \boldsymbol{x} 
\end{equation*}
where $P$ is the solution of the matrix differential equation
\begin{equation}  \label{mdre_p_lqr}
\dot{P}+PA+A^TP+Q- \frac{1}{2} PBR^{-1}B^{T}P -\frac{1}{4} PBR^{-1}B^{T}P^T - \frac{1}{4}  P^TBR^{-1}B^{T}P \ =0 
 \end{equation}
 with the final value $P(t_f)=F(t_f) \label{p_eqn_1_f} $.
\end{enumerate}

 \end{corollary}
 \begin{proof}
 \normalfont
 Clearly, if $P$ satisfies \eqref{mdre_p_lqr}, then $s$ is independent of $\boldsymbol{x}$ and the obtained control law indeed results in an admissible process. Hence, selected $q$ is the solving function. Furthermore, for minimization it required that the second term in \eqref{seqn} is zero which gives the optimal control law as:
 \begin{equation*}
\boldsymbol{u}^*=- \frac{1}{2} {R}^{-1}B^T \left(  P^T +P \right)\boldsymbol{x} 
 \end{equation*}

 \end{proof}
 
 \subsubsection{Solution of OCP\ref{prob_lqr} with final time $t_f \rightarrow \infty$(Infinite Horizon LQR)} \label{inf_lqr}
 Next, the infinite final time LQR problem for a linear time invariant (LTI) is considered. For this case, the terminal  state weighing matrix $F(t_f)=0$.
 The problem statement now becomes:
\begin{problem}
 (Infinite Horizon LQR problem):\\ 
{Compute an optimal control law $\boldsymbol{u}^*(t)$ which minimizes the quadratic performance index/cost functional:
\begin{align*}
J(\boldsymbol{x}(t),\boldsymbol{u}(t),t)=  0.5 \Big[ \int_{t_0}^{ \infty} \boldsymbol{x}^T(t)Q\boldsymbol{x}(t)+\boldsymbol{u}^T(t)R\boldsymbol{u}(t) dt \Big]
\end{align*}
subject to the system dynamics
$ \dot{\boldsymbol{x}}(t)=A\boldsymbol{x}(t)+B\boldsymbol{u}(t)  $
and drives the states of system to zero (Regulation).
Here $ \boldsymbol{x}(t_0)=\boldsymbol{x}_0$ is given, and $\boldsymbol{x}(\infty)$ is free.
Also, $Q \succeq0 \text{ and } R  \succ 0 $.
}
 \end{problem}
 \begin {solution*}
 \normalfont

 For this problem, the matrix differential equation \eqref{mdre_p_lqr} needs to be solved with the boundary condition $P(\infty)=0$. Computing this solution is equivalent to solving the algebraic equation:
 \begin{equation} \label{ae_p_lqr}
 PA+A^TP+Q-\frac{1}{2} PBR^{-1}B^{T}P -\frac{1}{4} PBR^{-1}B^{T}P^T - \frac{1}{4}  P^TBR^{-1}B^{T}P \ =0 
 \end{equation}
 and the resulting optimal control law is  given as:
 \begin{equation} \label{u_lqr_inf_1}
 \boldsymbol{u}^*=- \frac{1}{2} {R}^{-1} B^T\left(  P^T +P \right)\boldsymbol{x} 
\end{equation}

Similar to the finite-final time case it is easy to verify that the function $q=\boldsymbol{x}^TP\boldsymbol{x}$ is a solving function. Note that since the final time $t_f \rightarrow \infty$, it is necessary to ensure the stability of closed loop system. The next lemma provides a proof of the closed-loop stability under another condition on the $P$ matrix.

\begin{lemma} \label{lm_1}
For the system $\dot{\boldsymbol{x}}=A\boldsymbol{x}+B\boldsymbol{u}$, the control input \eqref{u_lqr_inf_1} ensures the stability of closed loop if $P$ satisfies \eqref{ae_p_lqr} with $(P+P^T) \succ 0$.
\end{lemma}
\begin{proof}
\normalfont
Let the Lyapunov function be $V(\boldsymbol{x}(t))= \boldsymbol{x}^T(P+P^T)\boldsymbol{x}$.
Then :
\begin{align*}
\dot{V}&= \boldsymbol{x}^T (P+P^T) \dot{\boldsymbol{x}} +\dot{\boldsymbol{x}}^T (P+P^T) \boldsymbol{x}  \\
&= \boldsymbol{x}^T \left(  A^TP+A^TP^T+A^TP+A^TP^T \right)\boldsymbol{x} - 2\boldsymbol{x}^TPBR^{-1} B^TP\boldsymbol{x} -\boldsymbol{x}^TP^TBR^{-1} B^T P\boldsymbol{x}\\
& \hspace{0.5 cm} -\boldsymbol{x}^TPBR^{-1} B^T P^T \boldsymbol{x}  \\
&= 2\boldsymbol{x}^T\left[ -Q - \frac{1}{4} PBR^{-1}B^TP^T -\frac{1}{4} P^TBR^{-1}B^TP -\frac{1}{2} PBR^{-1}B^TP  \right]\boldsymbol{x}
\end{align*}
It is easy to verify the quantity $\left[ -Q - \frac{1}{4} PBR^{-1}B^TP^T -\frac{1}{4} P^TBR^{-1}B^TP -\frac{1}{2} PBR^{-1}B^TP \right] $ is negative semi-definite if $Q \succeq 0$ and negative definite if $Q \succ 0$. Thus, using Lyapunov theory \cite{khalil1996nonlinear}, the closed loop system is stable if $Q \succeq0$ and asymptotically stable if $Q \succ 0$.
\end{proof}
 \end{solution*}
 \begin{remark1}
 \normalfont
The matrix differential equation \eqref{mdre_p_lqr} reduces to the popular matrix differential riccati equation (MDRE) for a symmetric $P$ matrix.  Also, for the infinite final-time case, the algebraic equation \eqref{ae_p_lqr} admits more number of solutions than MDRE as demonstrated in Example $3$ in Section \ref{ne}. A more rigorous analysis and application of these solutions is the subject of future research.
 \end{remark1}
\subsection{ Solution of OCP \ref{prob_lqt} (LQT problem)}
The equivalent optimization problem for OCP \ref{prob_lqt} is given as:
\begin{EOCP} \label{eocp_2}  Compute an optimal control law $\boldsymbol{u}^*(t)$ which 
\begin{enumerate} [(i)]
\item { $\min\limits_{(\boldsymbol{x},\boldsymbol{u}) \in \mathbb{R}^n \times \mathbb{R}^m} s(\boldsymbol{x},\boldsymbol{u},t)$, $\forall t \in [t_0,t_f)$,  where
 \begin{equation*} 
 s=\frac{\partial q}{\partial t}+\frac{\partial q}{\partial \boldsymbol{x}}[A\boldsymbol{x}+B\boldsymbol{u}]+\boldsymbol{e}^TQ\boldsymbol{e}+\boldsymbol{u}^TR\boldsymbol{u}
 \end{equation*}
 }
\item{$\min\limits_{x \in \mathbb{R}} s_f(\boldsymbol{x}(t_f))$, where 
 \begin{equation*} 
 s_f= \boldsymbol{e}^T(t_f)F(t_f)\boldsymbol{e}(t_f) - q(\boldsymbol{x}(t_f),t_f)
  \end{equation*}
  }
\end{enumerate} 
\end{EOCP}

Another main result of the paper is given in the next proposition, which will be useful in computing the direct solution to OCP \ref{prob_lqt}.
\begin{proposition}  \label{prop_2}
For the Equivalent OCP \ref{eocp_2}, let the  Krotov function be chosen as 
\begin{equation} \label{q_lqt_cov}
q \triangleq \boldsymbol{x}^T P\boldsymbol{x} - 2 \boldsymbol{g}^{T}\boldsymbol{x} 
\end{equation}
and every element of the matrix $P$ and the vector $\boldsymbol{g}$ is differentiable $\forall \ t \in [t_0,t_f)$. Then, the following statements are equivalent:
\begin{enumerate}[(a)]
\item{ $s(\boldsymbol{x},\boldsymbol{u},t)$  and $s_f(\boldsymbol{x}(t_f))$ are \color{black}{bounded below and convex } in $(\boldsymbol{x}, \boldsymbol{u})$ and $\boldsymbol{x}(t_f)$ respectively;}
\item{ 
\begin{enumerate}[i)]
\item{ \label{itm:Prop2P} $P$ satisfies the following matrix inequalities:

\begin{enumerate}[1)]
\item { 
$
\dot{P}+PA+A^TP -\frac{1}{2} PBR^{-1}B^{T}P -\frac{1}{4} PBR^{-1}B^{T}P^T -  \frac{1}{4} P^TBR^{-1}B^{T}P+C^{T}QC \succeq 0 \\ \forall t \in [t_0,t_f)
$
}
\item {$C^T(t_f)F(t_f)C(t_f) -P(t_f) \succeq 0 $}
\end{enumerate}
}
\item{$\boldsymbol{g}$ satisfies the vector differential equation
\begin{equation*} \label{g_eqn}
\dot{\boldsymbol{g}}+A^T\boldsymbol{g}+C^TQ\boldsymbol{z} -\frac{1}{2} \left( P+P^T \right) BR^{-1}B^T\boldsymbol{g}= \boldsymbol{0},
\end{equation*}
with the boundary condition
$ \boldsymbol{g}(t_f)= C^T(t_f) F(t_f) \boldsymbol{z}(t_f). $
}
\end{enumerate}
}
\end{enumerate}
\end{proposition}
\begin{proof}
\normalfont{
 With $q$ chosen as in \eqref{q_lqt_cov}, the function $s(x,u,t)$ is given as
\begin{align*} 
 s&=\boldsymbol{x}^T \dot{P} \boldsymbol{x}-2\dot{\boldsymbol{g}}^{T}x+(x^T (P+P^T)-2\boldsymbol{g}^{T})[A\boldsymbol{x}+B\boldsymbol{u}] +\boldsymbol{e}^TQ\boldsymbol{e}+\boldsymbol{u}^TR\boldsymbol{u  }\\
 &=\boldsymbol{x}^T( \dot{P}+PA+P^TA +C^{T}QC)\boldsymbol{x}+\boldsymbol{x}^TPB\boldsymbol{u} +\boldsymbol{x}^TP^TB\boldsymbol{u}-2 \dot{\boldsymbol{g}}^{T}\boldsymbol{x}-2\boldsymbol{g}^{T}A\boldsymbol{x}-2\boldsymbol{g}^{T}B\boldsymbol{u}\\
 &\hspace{0.7 cm}+\boldsymbol{z}^TQ\boldsymbol{z}-2\boldsymbol{x}^{T}C^{T}Q\boldsymbol{z}+\boldsymbol{u}^TR\boldsymbol{u}
 \end{align*}
Adding and subtracting the terms $\boldsymbol{x}^T \left( \frac{1}{2} PBR^{-1}B^{T}P +\frac{1}{4} PBR^{-1}B^{T}P^T + \frac{1}{4}  P^TBR^{-1}B^{T}P \right)\boldsymbol{ }x$,\\ $\boldsymbol{g}^TBR^{-1}B^T\boldsymbol{g}$ and $\boldsymbol{x}^T (P+P^T)BR^{-1}B^{T}\boldsymbol{g}$ , we get
 \begin{multline}
 s= \boldsymbol{x}^T \big[ \dot{P}+PA+A^TP -\frac{1}{2} PBR^{-1}B^{T}P+C^{T}QC -\frac{1}{4} PBR^{-1}B^{T}P^T -\frac{1}{4} P^TBR^{-1}B^{T}P \big]\boldsymbol{x}\\
-2\boldsymbol{x}^{T} \big[\dot{\boldsymbol{g}} +A^T\boldsymbol{g}+C^{T}Q\boldsymbol{z}-\frac{1}{2} \big( P+P^T \big)  BR^{-1}B^T\boldsymbol{g} \big]+\frac{1}{2}\boldsymbol{u}^TB^TP\boldsymbol{x}-\boldsymbol{u}^TB^T\boldsymbol{g} +\frac{1}{2}\boldsymbol{u}^TB^TP^T\boldsymbol{x} \\
 - \frac{1}{2} \boldsymbol{x}^TPBR^{-1}B^T\boldsymbol{g} -\frac{1}{2} \boldsymbol{x}^TP^TR^{-1}B^T\boldsymbol{g}+\frac{1}{2} \boldsymbol{x}^TPB\boldsymbol{u}+ \frac{1}{4}\boldsymbol{x}^TPBR^{-1}B^TP\boldsymbol{x }+\frac{1}{2} \boldsymbol{x}^TP^TB\boldsymbol{u }\\
+\frac{1}{4}\boldsymbol{x}^TP^TBR^{-1}B^TP\boldsymbol{x }+ \frac{1}{4} \boldsymbol{x}^T P^TBR^{-1}B^TP\boldsymbol{x}+\frac{1}{4}\boldsymbol{x}^TP^TBR^{-1}B^TP^T\boldsymbol{x} \\
 - \frac{1}{2} \boldsymbol{g}BR^{-1}B^TP\boldsymbol{x}- \frac{1}{2}\boldsymbol{g}BR^{-1}B^TP^T\boldsymbol{x}+ \boldsymbol{g}^TBR^{-1} B^T\boldsymbol{g}\label{s_eqn_lqt_1}
\end{multline}
Since $R \succ 0$, there exists a {unique positive definite matrix} \cite{SteLie}, say $\tilde{R}$, such that
$ \tilde{R}^2=R \ \text{and} \ \tilde{R}^{-1} \tilde{R}^{-1}=R^{-1}  $.
Now rearranging the terms in \eqref{s_eqn_lqt_1}, we get:
\begin{multline} \label{s_eqn_lqt_2}	
 s=\boldsymbol{x}^T\left[ \dot{P}+PA+A^TP -\frac{1}{2} PBR^{-1}B^{T}P -\frac{1}{4} PBR^{-1}B^{T}P^T -  \frac{1}{4} P^TBR^{-1}B^{T}P+C^{T}QC \right] \boldsymbol{x}\\
 	+ \left[\tilde{R} \boldsymbol{u} + \frac{1}{2}\tilde{R}^{-1} B^{T}(P+P^T)\boldsymbol{x} -\tilde{R}^{-1}B^{T}\boldsymbol{g}\right]^{T} \left[\tilde{R}\boldsymbol{u}+\frac{1}{2}\tilde{R}^{-1}B^{T}(P+P^T)\boldsymbol{x} -\tilde{R}^{-1}B^{T}\boldsymbol{g}\right]\\
 -2\boldsymbol{x}^{T} \left[\dot{\boldsymbol{g}}+A^T\boldsymbol{g}+C^{T}Q\boldsymbol{z}-\frac{1}{2} \left( P+P^T \right) BR^{-1}B^T\boldsymbol{g} \right] 
 + \left[\boldsymbol{z}^TQ\boldsymbol{z}-\boldsymbol{g}^{T}BR^{-1}B^{T}\boldsymbol{g} \right]  
\end{multline}

The second term in \eqref{s_eqn_lqt_2} is positive definite and strictly convex. The third term is linear in $\boldsymbol{x}$, and the fourth term is independent of $\boldsymbol{x}$ and $\boldsymbol{u}$. The first term in \eqref{s_eqn_lqt_2} is quadratic in $\boldsymbol{x}$ and it is convex iff it is positive semi-definite. Hence, the function $s(\boldsymbol{x},\boldsymbol{u},t)$ is convex iff
\begin{equation*} 
\dot{P}+PA+A^TP -\frac{1}{2} PBR^{-1}B^{T}P -\frac{1}{4} PBR^{-1}B^{T}P^T - \frac{1}{4} P^TBR^{-1}B^{T}P+C^{T}QC \succeq 0 \ \forall t \in [t_0,t_f).
\end{equation*}
Since the second term in \eqref{s_eqn_lqt_2} is linear in $\boldsymbol{x}$, it is bounded below if
\begin{equation} \label{g_eqn_lqt}
\dot{\boldsymbol{g}}+A^T\boldsymbol{g}+C^TQ\boldsymbol{z} - \frac{1}{2}(P+P^T) BR^{-1}B^T\boldsymbol{g}=0. 
\end{equation}
Finally, the term independent of $\boldsymbol{x}$ and $\boldsymbol{u}$, i.e., $\boldsymbol{z}^TQ\boldsymbol{z} -\boldsymbol{g}^{T}BR^{-1}B^{T}\boldsymbol{g}$ is bounded because $\boldsymbol{z}(t)$ and $\boldsymbol{g}(t)$ are bounded \cite{anderson2007optimal}. Next, with $q(\boldsymbol{x}(t),t)$ as in \eqref{q_lqt_cov}, $s_f(\boldsymbol{x}(t_f))$ is given as
\begin{align*}
s_f = &\boldsymbol{x}^T(t_f) \big[C^T(t_f)F(t_f) C(t_f)-P(t_f)\big]\boldsymbol{x}(t_f) +\big[2 \boldsymbol{g}^{T}(t_f)- 2 \boldsymbol{z}^T (t_f) F(t_f) C(t_f)\big] \boldsymbol{x}(t_f)\\
& \hspace{0.5 cm}+\boldsymbol{z}^T(t_f) F(t_f) \boldsymbol{z}(t_f)
\end{align*}
Similar to the case of $s$, $s_f$ is convex iff
\begin{equation*} \label{p_con_f_lqt}
C^T(t_f) F(t_f)C(t_f)- P(t_f) \succeq 0
\end{equation*}
Finally, $s_f$ is bounded below if 
\begin{equation*}
\boldsymbol{g}(t_f)= C^T(t_f) F(t_f) \boldsymbol{z}(t_f)
\end{equation*}
}
\end{proof}
The following corollary computes the optimal control law for OCP \ref{prob_lqt}.
\begin{corollary} \label{cor_2}
With $q(\boldsymbol{x},t)$ selected as in Proposition \ref{prop_1}, the following statements are true:
\begin{enumerate}
\item The function $q=\boldsymbol{x}^TP\boldsymbol{x}-2g^T\boldsymbol{x}$ is a solving function for OCP \ref{prob_lqt}.
\item The global optimal control law for OCP \ref{prob_lqt} is given by:
\begin{equation*} \label{u_eqn}
\boldsymbol{u}^*=- \frac{1}{2} {R}^{-1} B^T\left( P^T +P \right)\boldsymbol{x} +R^{-1}B^T\boldsymbol{g}
\end{equation*}
where $P$ is the solution of the matrix differential equation
\begin{equation} \label{mdre_p_lqt}
\dot{P}+PA+A^TP -\frac{1}{2} PBR^{-1}B^{T}P -\frac{1}{4} PBR^{-1}B^{T}P^T - \frac{1}{4} P^TBR^{-1}B^{T}P+C^{T}QC=0
\end{equation}
with the final value $P(t_f)=C^T(t_f) F(t_f)C(t_f) \label{p_eqn_1_f} $ and $\boldsymbol{g}(t)$ satisfies the differential equation $\dot{\boldsymbol{g}}+A^T\boldsymbol{g}+C^TQ\boldsymbol{z }- \frac{1}{2}(P+P^T) BR^{-1}B^T\boldsymbol{g}=0$ with final value $\boldsymbol{g}(t_f)= C^T(t_f) F(t_f) \boldsymbol{z}(t_f)$.
\end{enumerate}
\end{corollary} 

\begin{proof}
\normalfont{ Clearly, if $P$ satisfies \eqref{mdre_p_lqt} and $\boldsymbol{g}$ satisfies \eqref{g_eqn_lqt}, then $s$ is independent of $\boldsymbol{x}$ and the obtained control law indeed yields an admissible process. Hence, $q(\boldsymbol{x}(t),t)$ is the solving function. The third term in \eqref{s_eqn_lqt_2} is strictly convex and attains a minimum value when
\begin{align*}
\tilde{R} \boldsymbol{u }+ \frac{1}{2}\tilde{R}^{-1} B^{T}(P+P^T)\boldsymbol{x }-\tilde{R}^{-1}B^{T}\boldsymbol{g}= \boldsymbol{0} \\
\implies \boldsymbol{u }=- \frac{1}{2} {R}^{-1} B^T\left( P^T +P \right)\boldsymbol{x }+R^{-1}B^T\boldsymbol{g}
\end{align*}
}
\end{proof}

\subsubsection{Solution of OCP\ref{prob_lqt} with final time $t_f \rightarrow \infty$(Infinite Horizon LQT)}
Next, the infinite final time LQT problem for a linear time invariant (LTI) is considered. For this case, the terminal cost $F(t_f)=0$.
The problem statement now reads:
\begin{problem}
(Infinite Horizon LQT problem):\\ 
{Compute an optimal control law $\boldsymbol{u}^*(t)$ which minimizes the quadratic performance index/cost functional:
\begin{align*}
J(\boldsymbol{e}(t),\boldsymbol{u}(t),t)= 0.5 \Big[ \int_{t_0}^{ \infty} \boldsymbol{e}^T(t)Q\boldsymbol{e}(t)+\boldsymbol{u}^T(t)R\boldsymbol{u}(t) dt \Big]
\end{align*}
subject to the system dynamics
$ \dot{\boldsymbol{x}}(t)=A\boldsymbol{x}(t)+B\boldsymbol{u}(t) $; $\boldsymbol{y}(t)=C\boldsymbol{x}(t)$
and drives the states of system to a desired trajectory $\boldsymbol{z}(t)$. 
Here
$\boldsymbol{e}(t) \triangleq \boldsymbol{y}(t)-\boldsymbol{z}(t)$ is the error vector, $\boldsymbol{x}(t_0)=\boldsymbol{x}_0$ is given, $\boldsymbol{x}(t_f)$ is free and $t_f$ is fixed. $ \boldsymbol{x}(t_0)=\boldsymbol{x}_0$ is given, and $\boldsymbol{x}(\infty)$ is free.
Also, $Q \succeq0 \text{ and } R \succ0. $
}
\end{problem}
\begin {solution*}
\normalfont

For this problem, similar to the case of regulation, the matrix differential equation \eqref{mdre_p_lqt} needs to be solved with the boundary condition $P(\infty)=0$. Similarly, the differential equation \eqref{g_eqn_lqt} needs to be solved with boundary condition $\boldsymbol{g}(\infty)=\boldsymbol{0}$. Computing these solutions is equivalent to solving :
\begin{equation*} 
PA+A^TP+Q-\frac{1}{2} PBR^{-1}B^{T}P -\frac{1}{4} PBR^{-1}B^{T}P^T - \frac{1}{4} P^TBR^{-1}B^{T}P \ =0 
\end{equation*}
and
\begin{equation} \label{g_lqt_eqn}
\boldsymbol{g}(t)=- \int_{t}^{\infty} exp^{\left[A^{T}-\frac{1}{2}(P+P^T)BR^{-1}B^{T} (\tau -t)\right]} C^{T} Q \boldsymbol{z}(\tau) d \tau
\end{equation}
and the resulting optimal control law is given as:
\begin{equation*}
\boldsymbol{u}^*=- \frac{1}{2} {R}^{-1} B^T \left( P^T +P \right)\boldsymbol{x }+R^{-1}B^T\boldsymbol{g}
\end{equation*}
Similar to the finite-final time case it is easy to verify that the function $q=\boldsymbol{x}^TP\boldsymbol{x}-2\boldsymbol{g}^T\boldsymbol{x}$ is a solving function. In this case too, similar to the case of regulation, the stability of closed loop is ensured if $(P+P^T) \succ 0$.
\end{solution*}

\begin{remark1}
\normalfont
As is clear from the solution methodologies proposed above, the solution of the considered problems does not require the notion of co-sates and any a-prior information regarding the optimal cost function. Rather, whole solution technique depends upon the selection of Krotov function $q(\boldsymbol{x}(t),t)$ which directly affects the nature of equivalent optimization problems.
\end{remark1}

Before demonstrating the developed methodology with numerical examples, a brief overview of Krotov method, which is an iterative method to solve equivalent optimization problem in Theorem \ref{thm2} is provided. For more details refer \cite{Kro95, gurman2016certain} and the references therein.

\subsection{Krotov method} Krotov method is one of the iterative methods to solve the equivalent optimization problems of Theorem \ref{thm2}. The steps of this method are as below:
\begin{enumerate}
\item {Choose \textit{any} admissible control $\boldsymbol{u}_0(t)$ and compute the corresponding admissible process $\boldsymbol{v_0}=\left[\boldsymbol{u}_0(t),\boldsymbol{x}_0(t) \right]$ using the dynamical equation $\dot{\boldsymbol{x}}=f(\boldsymbol{x}(t),\boldsymbol{u}(t),t)$ of the system. Also compute the cost $J_0$. }
\item {\textit{Construct} a function $\phi(\boldsymbol{x}(t),t)$ such that :
\begin{enumerate}
\item $ s(t,\boldsymbol{x}_0(t),\boldsymbol{u}_0(t))= \max_{\boldsymbol{x} \in \mathbb{R}^n } \ s(t,\boldsymbol{x}(t),\boldsymbol{u}(t)) $
\item $ s_f(\boldsymbol{x}_0(t_f)) =\max_{\boldsymbol{x} \in \mathbb{X}_f } \ s_f(\boldsymbol{x}(t_f))$
\end{enumerate}
}
\item Construct the control input $\tilde{\boldsymbol{u}}(t)$ such that:
\begin{equation*}
\tilde{\boldsymbol{u}}(t)= \arg_{\boldsymbol{u}} \min s(t,\boldsymbol{x},\boldsymbol{u})
\end{equation*}
\item Compute the cost of resulting admissible process and compare it with the cost obtained in step $1$ to determine the extent of improvement achieved viz. compare the difference between the two costs with a predefined tolerance, say $\epsilon$. If difference is less than $\epsilon$ go back to step $2$ with improved process obtained in $3$, otherwise stop.
\end{enumerate}
Now, the optimal control problem for a scalar system is solved to clarify the approach of Krotov method.

\begin{problem*}
For the system $\dot{x}(t)=-x(t)+u(t)$ with $x(0)=5$, compute the optimal control law so as to minimize $J= \int_0^{10} x^2(t) +u^2(t) dt $.
\end{problem*}
\begin{solution*}
\normalfont
\begin{enumerate}[i)]
\item Choose an admissible process as :
$u_0(t)= 0$ which gives $x_0(t)= 5$. The corresponding cost is computed to be :
$ J_0=25$.
\item Next a function $\phi$ is chosen so as to maximize:
$s(t,x,u_0)= \frac{\partial \phi} {\partial t} + \frac{\partial q} {\partial x} \left(-x+u \right)+ x^2+u^2 $ over $x$. Upon a quadratic selection of $\phi$ as $\phi=px^2$, the function $s$ is given as :
$s(t,x,u_0)= x^2( \dot{p} -2p+1) $ and $g(x(10))=- p(10)x^2(10)$. Choosing $p$ such that
$\dot{p}-2p+1=0$ and $p(10)=0$ clearly turns out to be appropriate selection for this case.
Then the improved process is given as $v_1= (x_1,u_1)$ where:
$ \dot{x_1}= (-1-p) x_1$ and $p(10)=0$. The corresponding cost is computed to be:
$J_1= 10.42$.
\item Next, the improving function is again selected to be a quadratic function $q=px^2$ and the corresponding equations for $p$ are given as :
\begin{equation*}
\dot{p}-2p+1-p^2=0
\end{equation*}
with $p(10)=0$.
The corresponding improved process is given by:\\
$\dot{x}_2=(-1-p)x_2$ with $u_2=-p x_2$ and the corresponding cost is computed as $J_2=10.36$.
\end{enumerate}
It was found that upon further iterations no improvement of cost was achieved leading to the conclusion that this process is the globally optimal process.
\end{solution*}

\section{Numerical Examples} \label{ne}
In this section, the proposed solution methodology is demonstrated for computing the finite and infinite horizon optimal control laws.
\begin{example} \label{ex1}
For the  scalar system described by the dynamics
$ \dot{x}=-x+u $ with $x(t_0)=x_0$; where $x \in \mathbb{R}$ and $ u \in \mathbb{R}$
find the optimal control law so as to minimize
$ J= \frac{1}{2} \int_{t_0}^{\infty} \big(x^2+u^2\big) dt. $
\end{example}
\begin{solution*}
\normalfont
The optimization problem to be solved as per Theorem \ref{thm1} is as follows:
 \begin{equation*}
 \min\limits_{(x,u) \in \mathbb{R} \times \mathbb{R}}  s({x},{u},t)
 \end{equation*} 
 where
  \begin{equation*}
  s=\frac{\partial q}{\partial t}+\frac{\partial q}{\partial x}[-x+u]+x^2+u^2
  \end{equation*} 
Let the Krotov function be 
\begin{equation*}
q(x)=px^2, \text{ }p>0
\end{equation*}
Then the function $s(x,u,t)$ is given as:
\begin{equation*} \label{s_01}
s(x,u,t)=x^2(-2p+1) +2pxu +u^2
\end{equation*}
This function is nonlinear and non-convex, which is generally solved using an iterative solution procedure for {any} value of $p$. However, using Proposition \ref{prop_1} the direct solution can be obtained by choosing $p$ as so as to satisfy :
\begin{equation*}
\label{p_ineq_1}
2p-p^2+1 \geq 0
\end{equation*}
which gives $0 < p \leq 0.414$.
The plots of this function for different values of $p$ are given in Figure \ref{fig_1}. Finally, to minimize $s(x,u,t)$, $p$ is selected so as to satisfy $2p-p^2+1=0$ yielding $p=0.414$.
\begin{figure}[h]
\begin{subfigure}[b]{0.55\textwidth}
   \includegraphics[width=0.85\linewidth]{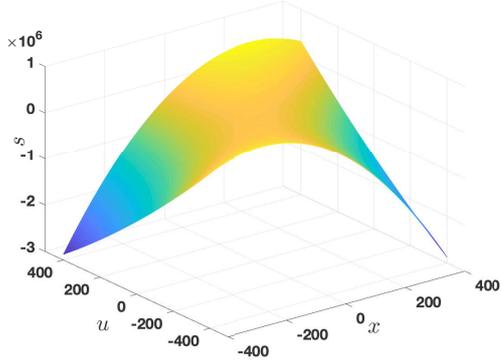}
   \caption{$p=5$}
   \label{fig_1:a} 
\end{subfigure}
\begin{subfigure}[b]{0.55\textwidth}
   \includegraphics[width=0.85\linewidth]{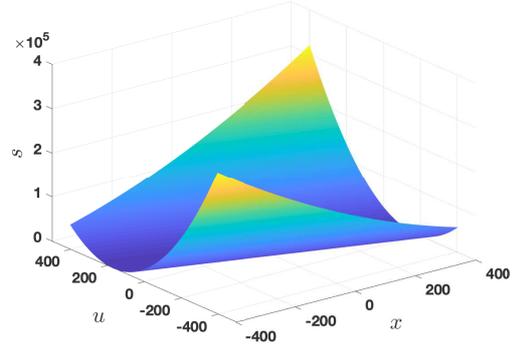}
   \caption{$p=0.414$}
   \label{fig_1:b}
\end{subfigure}
\caption{Plots of the function $s$ for different values of $p$ (Example \ref{ex1})}
\label{fig_1}	
\end{figure}
Thus, the optimal control law is computed using Corollary  \ref{cor_1} as
\begin{equation*}
u^*(t)=-px^*(t)= -0.414 x^*(t).
\end{equation*}
\end{solution*}

\begin{example} \label{ex2}
For the system described by the dynamics:
$\dot{x}=ax+bu$ with $x(t_0)=x_0$ find optimal control law so that output
$y=cx$ tracks the reference trajectory $z=\alpha sin(\omega t)$ and  the performance functional given by $J= \frac{1}{2} \int_{t_0}^{\infty} \big(me^2+nu^2\big) dt $
is minimized. Here, $e$ is the error defined as $e\triangleq z-y$. Here $x \in \mathbb{R}$, $ u \in \mathbb{R}$ and $e \in \mathbb{R}$.
\end{example}
\begin{solution*}
\normalfont
The optimization problem which is to be solved is given as:
\begin{equation*}
\min\limits_{(x,u) \in \mathbb{R}^n \times \mathbb{R}^m}  s({x},{u},t)
\end{equation*}
, where
  \begin{equation*}
  s=\frac{\partial q}{\partial t}+\frac{\partial q}{\partial x}[ax+bu]+me^2+nu^2
  \end{equation*} 
We use the Krotov function as :
\begin{equation*} \label{q_lqt}
q(x,u) =px^2 - 2 gx   \text{ where }  p  \succ 0
\end{equation*}
Then, the function $s(x,u,t)$ is given as:
\begin{align*}
s= -2 \dot{g}x+&2apx^2+2pbxu+-2gax-2gbu \\
&+mz^2+mc^2x^2-2mcxz+nu^2
\end{align*}
Clearly, the characteristics of the function $s(x,u,t)$ depend upon $p$ and $g$ and in general it is a non-convex function.
Again, it is easily verified that the function $s(x,u,t)$ is indeed convex for the selection which satisfies the conditions of Proposition \ref{prop_2}. 
Specifically,
\begin{enumerate}
\item $p$ is selected such that 
\begin{align*}
&2apn-{p^2b^2}+mnc^2 \geq 0; \ p>0
\end{align*}
\item $g(t)$ is a time varying function which can be computed as the steady state solution of the vector differential equation in Proposition \ref{prop_2}:
\begin{align*}
g(t)&= \alpha c m \int_{t}^{\infty} e^{(a-\frac{pb^2}{n})(\tau -t)} sin(\omega \tau) d \tau \\
&= \beta  \left[ -(an-{pb^2}) sin (\omega t) +n \omega cos(\omega t) \right ]
\end{align*}
where $\beta =\dfrac{\alpha   c m n}{\left ( na- {b^2p}\right)^2+n^2\omega^2}$
\end{enumerate}
Finally, the optimal control law is given as :
\begin{align*}
u^*=  -\frac{bpx^*}{n} + \frac{b}{n} \beta  \left[ -(an-{pb^2}) sin (\omega t)+n \omega cos(\omega t) \right].
\end{align*}
\begin{figure}
\begin{subfigure}[b]{0.55\textwidth}
   \includegraphics[width=0.85\linewidth,  height=1000mm, keepaspectratio]{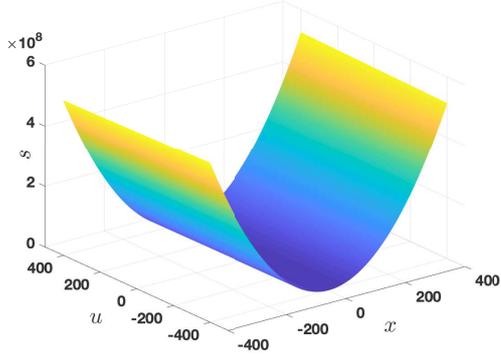}
   \caption{Plot of s function for $ p=17.98 $ (Example \ref{ex2})}
   \label{fig_s:a} 
\end{subfigure}
\begin{subfigure}[b]{0.55\textwidth}
   \includegraphics[width=0.85\linewidth]{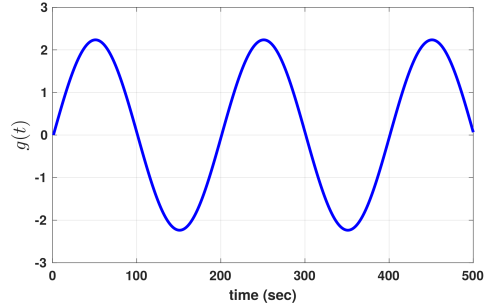}
   \caption{Plot of the function $ g(t) $}
   \label{fig_s:b}
\end{subfigure}
\caption{Plots of $s$ and $g(t)$}
\label{fig_2}	
\end{figure}
The numerical values considered for simulation purposes 
are 
in Table \ref{t_1}, where for positive values of $ a $ the system is unstable. For these values the range of $p$ to ensure convexity of $s$ as per Proposition \ref{prop_2} is computed as $0<p \leq 17.98$. Following Corollary \ref{cor_2}, $p$  is taken as $17.98$ to obtain the global optimal control law. The plot of function $s$ for $p=17.98$ is shown in Figure \ref{fig_s:a} which clearly shows convexity of $s$. The boundedness of function $g(t)$ is clear from Figure \ref{fig_s:b}. Finally, closed loop response tracking the given reference trajectory is shown in Figure \ref{fig_op}. 

\begin{figure}[h]
   \includegraphics[width=1\linewidth, height=1000mm, keepaspectratio]{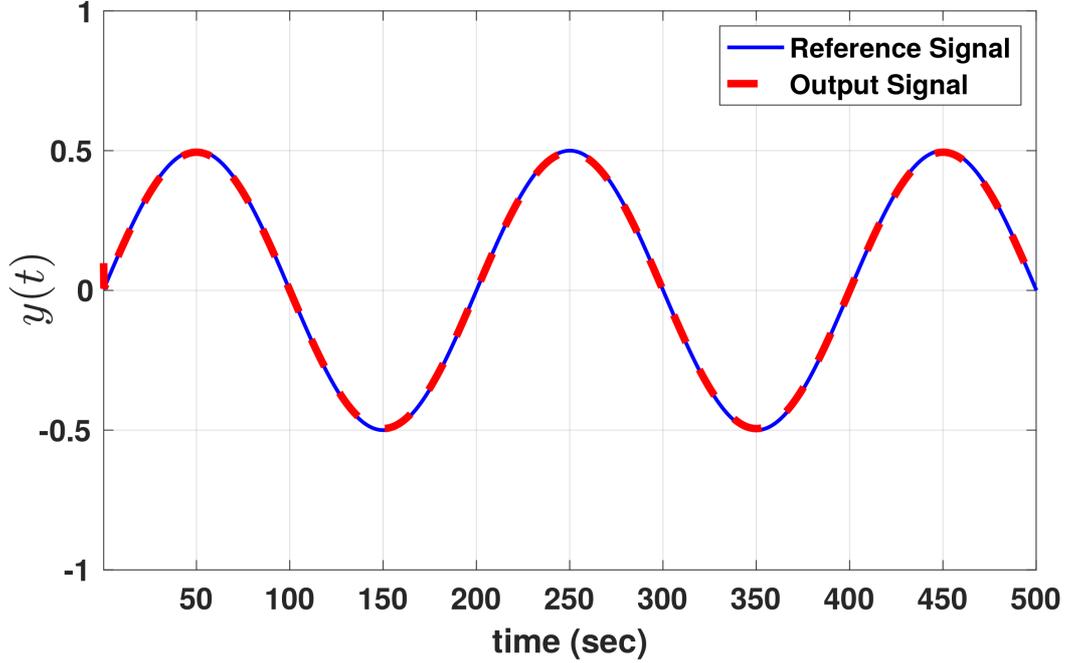}
\caption{Closed-Loop Response (Example \ref{ex2})}
\label{fig_op}	
\end{figure}
\begin{table}
\centering
\begin{tabular} {|c|c|c|c|}
\hline
\textbf{S. No.} &\textbf{Parameter}&\textbf{Value}&\textbf{Remarks}\\
\hline
$1$&$a$ & $1$ & Open loop unstable system\\
\hline
$2$&$b$& $1$ & -\\
\hline
$3$&$c$& $4$ & -\\
\hline
$4$&$\omega$ & $0.01 \pi$ & Angular frequency of reference \\
\hline
$5$&$\alpha$ & $0.5$ & Amplitude of reference signal \\
\hline
$6$&$m$ & $200$ & Weight assigned to tracking error\\
\hline
$7$&$n$ & $0.1$ & Weight assigned to control  input \\
\hline
$8$&$x_0$& 2& Initial condition\\
\hline
\end{tabular}
\caption{Parameters used in Simulation Results for Example \ref{ex2}}
\label{t_1}
\end{table}
\end{solution*}

\begin{example} \label{ex3}
Compute the optimal control law for the system 
$ \dot{x}= -\left( \frac{1}{t+1} \right)x +u $
which minimizes the performance index 
$ J= 0.5 x^2(t_f) + 0.5 \int_{0}^{t_f} \left( x^2 +u^2\right) dt $
with $t_f=5$ and $x(0)=20$.
\begin{solution*}
\normalfont
Following Theorem \ref{thm2}, the equivalent 
optimization problem to be solved is given as:
\begin{align} 
&\min\limits_{({x},{u}) \in \mathbb{R} \times \mathbb{R}} s({x},{u},t) \label{ex_1_siso_lqr_11} \\
& \min\limits_{{x(t_f)} \in \mathbb{R}} s_f({x(t_f)},t_f) \label{ex_1_siso_lqr_21}
\end{align} 
where 
$ s= \frac{\partial q}{\partial t}+ \frac{\partial q}{\partial x}\left[ \left( -\frac{1}{t+1} \right)x +u \right]+x^2+u^2 $ and
$ s_f= x^2(t_f)-q(x(t_f),t_f) .$
According to Proposition \ref{prop_1}, the Krotov function is chosen as:
\begin{equation*}
q=p(t)x^2
\end{equation*}
then the functions $s$ and $s_f$ are read as:
\begin{align*}
&s= \dot{p}(t)x^2-\left(\frac{2p(t)x^2}{t+1}\right)+2 p(t)xu+x^2+u^2 \\
&s_f= x^2(t_f) -p(t_f) x^2(t_f)
\end{align*}
Clearly the function $s$ is nonlinear and non-convex. Due to this, the equivalent problems \eqref{ex_1_siso_lqr_11}-\eqref{ex_1_siso_lqr_21} are solved using an iterative method, such as Krotov method in which $q$ is chosen appropriately at each iteration. Instead, using Proposition \ref{prop_1} the direct solution can be computed by choosing $p(t)$ such that:
\begin{equation*}
\dot{p}(t)- \left(\frac{2}{t+1} \right)p(t) -p^2(t) +1 \geq 0 \text{, } 1- p(5) \geq 0 \ \forall \ t
\end{equation*}
Finally, using Corollary \ref{cor_1}, $p(t)$ is chosen such that $\dot{p}(t)- \left(\frac{2}{t+1} \right)p(t) -p^2(t) +1 =0$ and $p(5)=1$. The solution of the above differential equation is given as
\begin{equation*}
p(t)= \frac{(2.423*10^5)t - t e^{2t}+ (4.846*10^5)}{(2.423*10^5)t + e^{2t} + te^{2t} + (2.423*10^5)}
\end{equation*}
and the optimal control law is computed as 
\begin{equation*} 
u^*= -p(t) x^*
\end{equation*}
The plots of $p(t)$, the state response of the closed-loop system and the optimal control input are shown in Figures \ref{p_fin_siso_reg} and \ref{x_fin_siso_reg} respectively.
\end{solution*}

\begin{figure}
\centering
\begin{minipage}{.5\textwidth}
\centering
\includegraphics[width=1\linewidth]{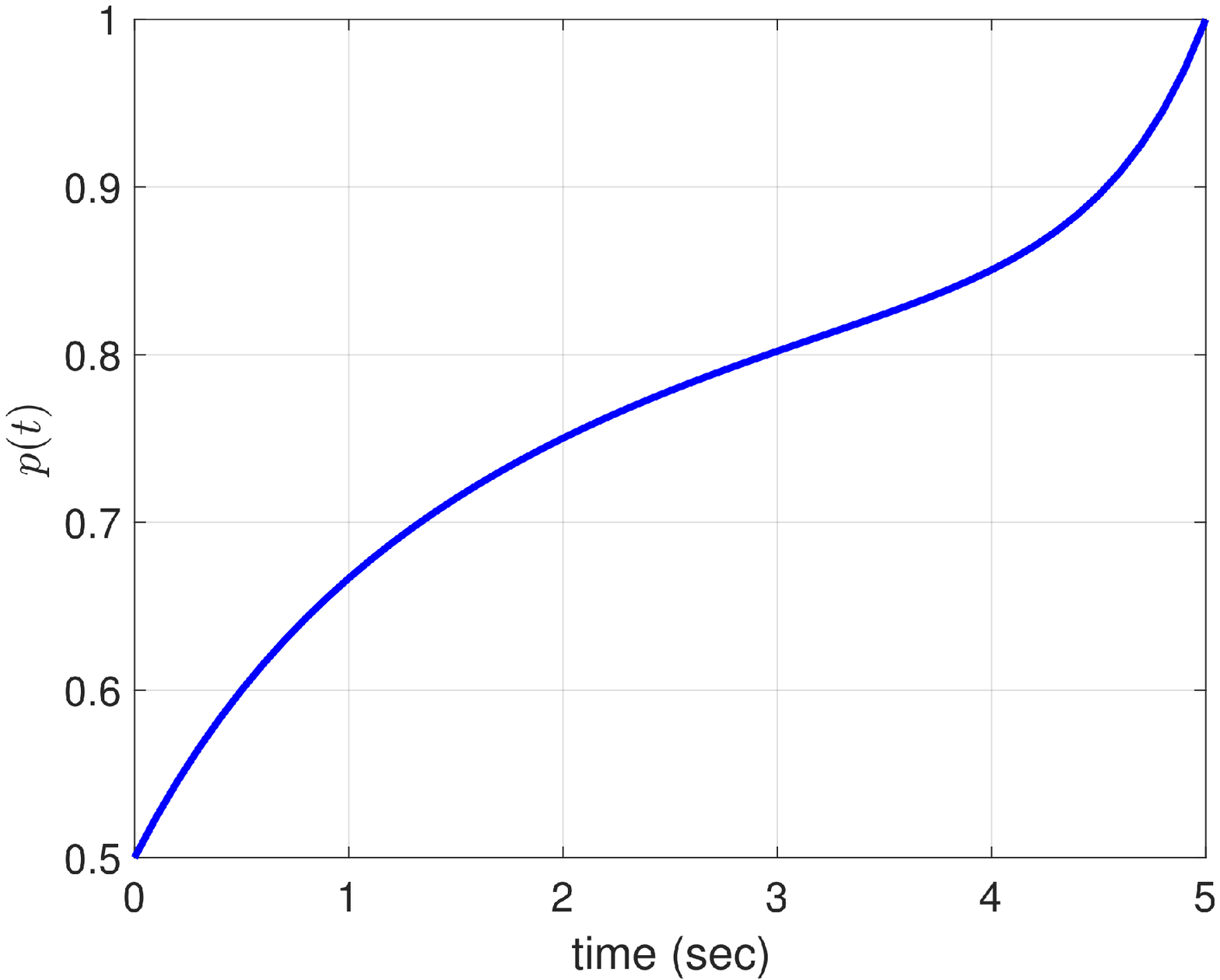}
\subcaption{Plot of $p(t)$}
\label{p_fin_siso_reg}
\end{minipage}%
\begin{minipage}{.5\textwidth}
\centering
\includegraphics[width=1\linewidth]{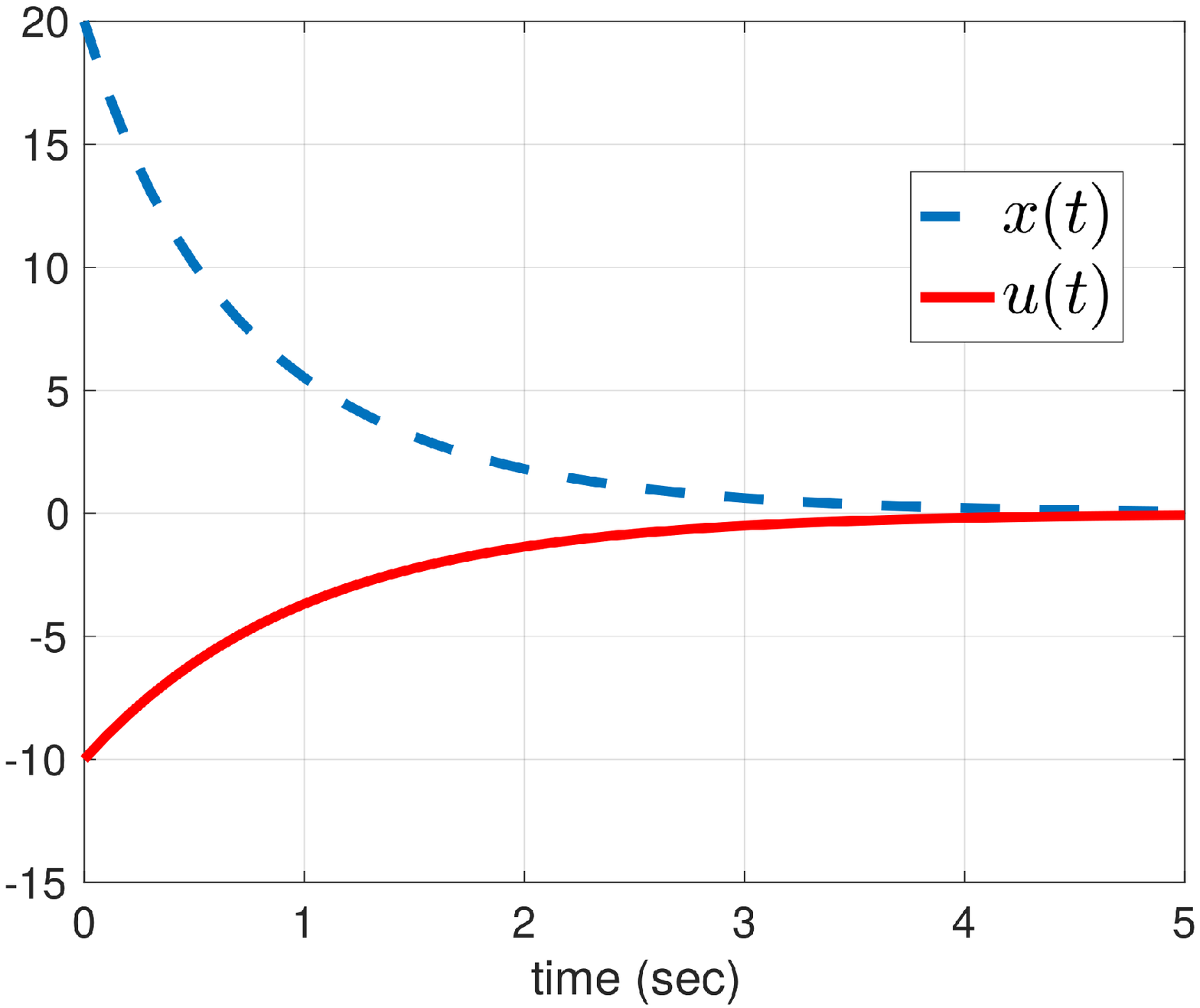}
\subcaption{Optimal control input and closed loop response}
\label{x_fin_siso_reg}
\end{minipage}
\caption{Plots of $p(t)$, control input and closed loop response for Example \ref{ex3}}
\end{figure}
\end{example}

\begin{example} \label{ex4}
Compute the optimal control law for the system given as:
$ \dot{x}= -\left[ \frac{1}{t+1} \right] x+u $
which minimizes the performance index 
$ J= 0.5 \left[10 e^2(t_f) + \int_{0}^{t_f} \left(1000e^2 +u^2 \right) dt \right] $
with $e(t) \triangleq z(t)-y(t)$, $z(t)$ is the reference trajectory- $z(t)=t$, $t_f=5$ and $x(0)=10$.
\begin{solution*}
\normalfont
The optimization problem to be solved is given as:
\begin{align*} 
&\min\limits_{({x},{u}) \in \mathbb{R} \times \mathbb{R}} s({x},{u},t) \\
& \min\limits_{{x(t_f)} \in \mathbb{R}} s_f({x(t_f)},t_f) 
\end{align*} 
where 
\begin{align*}
s= &\frac{\partial q}{\partial t}+ \frac{\partial q}{\partial x}\left[ \left( -\frac{1}{t+1} \right)x +u \right]+1000e^2+u^2 \text{ and }\\
& s_f= 10e^2(t_f)-q(x(t_f),t_f)
\end{align*}
Let the Krotov function be chosen as:
\begin{equation*}
q=p(t)x^2-2g(t)x
\end{equation*}
then the functions $s$ and $s_f$ are given as below:
\begin{align*}
&s=\dot{p}(t)x^2-2\dot{g}(t)x- \frac{2p(t)x^2}{t+1}+2p(t)xu+\frac{2g(t)x}{t+1}-2g(t)u+1000t^2+1000x^2-2000zx \\
&s_f= 10x^2(t_f)+10z^2(t_f)-2x(t_f)z(t_f) -p(t_f)x^2(t_f)+2g(t_f)x(t_f)
\end{align*}
Clearly the functions $s$ and $s_f$ are nonlinear and non-convex. Again, using Proposition \ref{prop_2} the direct solution can be computed by choosing $p(t)$ and $g(t)$ such that:
\begin{equation*}
\dot{p}- \left(\frac{2}{t+1} \right)p -p^2 +1000 \geq 0 \text{and } 10- p(5) \geq 0 \forall \ t
\end{equation*}
and
\begin{equation*}
\dot{g}(t)-\frac{1}{t+1} g(t)+1000t- pg=0 \text{ with } g(5)=50
\end{equation*}
Finally, using Corollary \ref{cor_1}, $p$ is chosen such that $\dot{p}- \left(\frac{2}{t+1} \right)p -p^2 +1000 = 0 \text{ and } 10- p(5) = 0 $ and $\dot{g}(t)-\frac{1}{t+1} g(t)+1000t- pg=0 \text{ with } g(5)=50$. The obtained $p(t)$ and $g(t)$ are as shown in Figure \ref{fig_p_n_g_siso_trac}.

\begin{figure}
\centering
\begin{minipage}{.5\textwidth}
\centering
\includegraphics[width=1\linewidth]{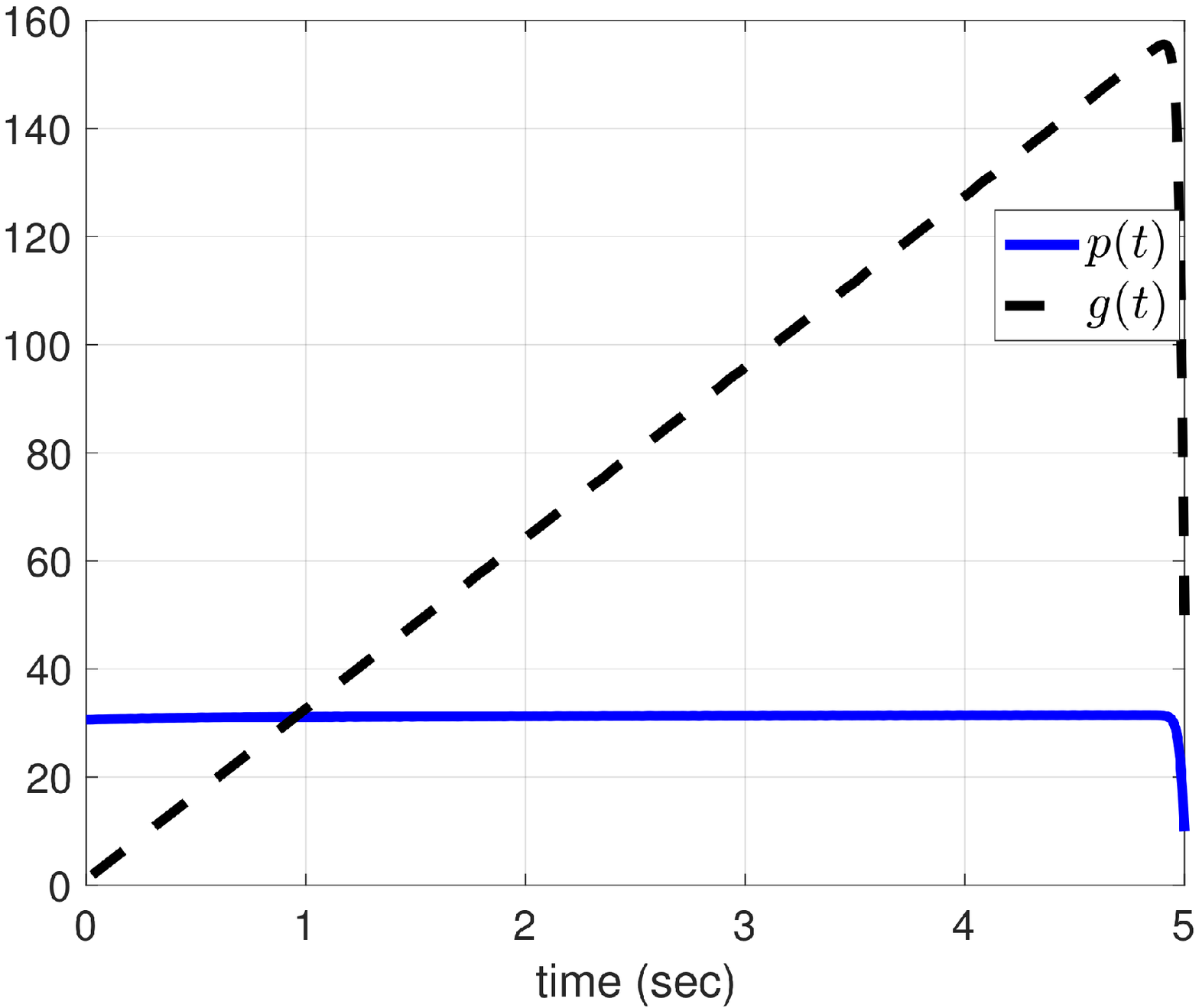}
\subcaption{Plots of $p(t)$ and $g(t)$}
\label{fig_p_n_g_siso_trac}
\end{minipage}%
\begin{minipage}{.5\textwidth}
\centering
\includegraphics[width=1\linewidth]{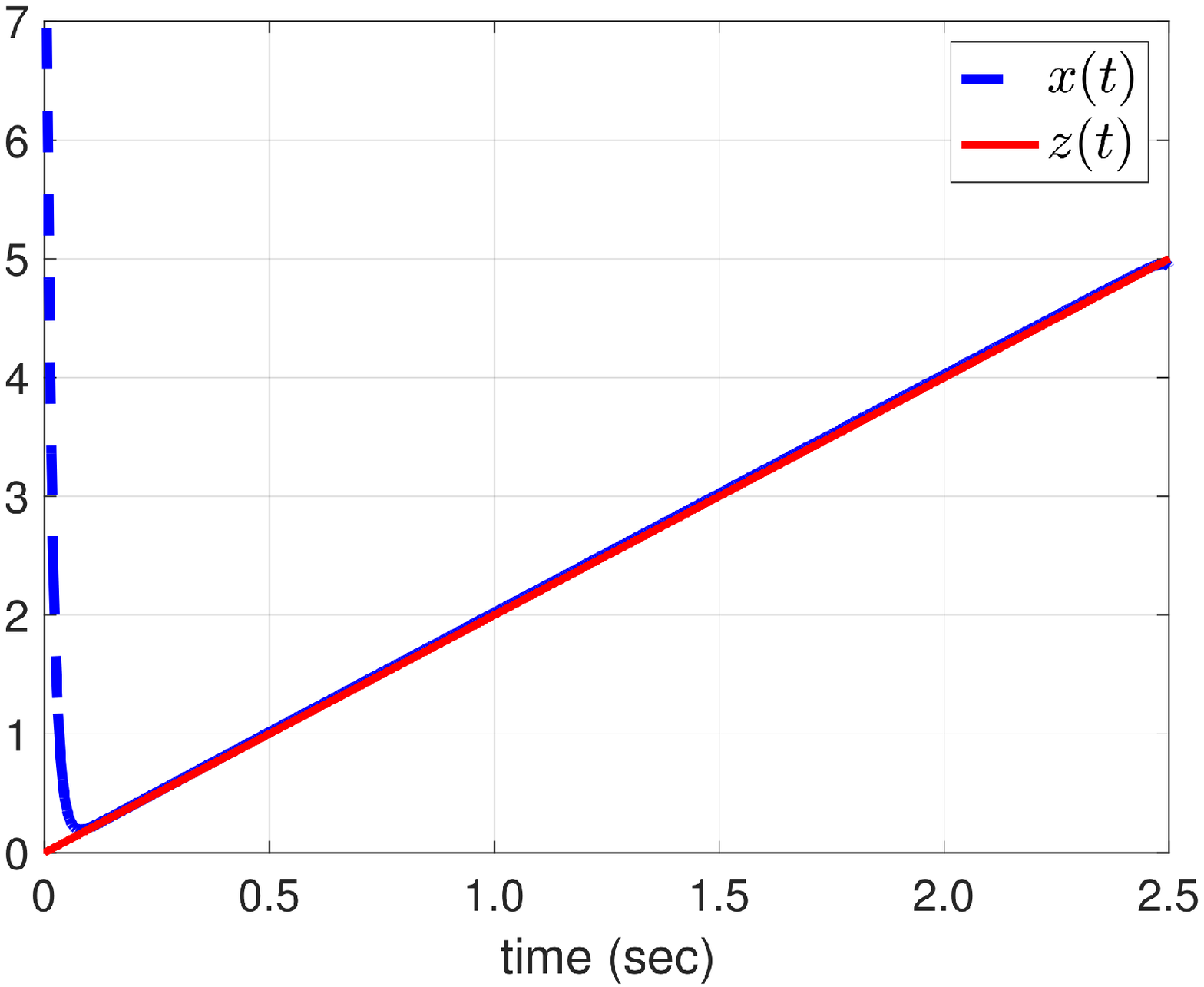}
\subcaption{Closed loop response}
\label{x_n_z_fin_siso_trac}
\end{minipage}
\caption{Plots of $p(t)$, $g(t)$ and closed loop response for Example \ref{ex4}}
\end{figure}

Finally, the optimal control law is computed as 
$ u^*= -p(t) x^*+g(t)$. The closed loop response is as shown in Figure \ref{x_n_z_fin_siso_trac}.
\end{solution*}
\end{example}
\normalfont{
Next, infinite-horizon optimal control problems for MIMO LTI system are considered.}
\begin{example} \label{ex5}
For the MIMO system:
$\dot{\boldsymbol{x}}= A \boldsymbol{x}+ B\boldsymbol{u} $
compute an optimal control law to minimize the performance index:
$ J= 0.5\int_{0}^{\infty} \left(\boldsymbol{x}^T Q \boldsymbol{x} + \boldsymbol{u}^T R \boldsymbol{u}\right) \ dt $
where
$ A= \begin{bmatrix}
0 & 1 \\
1 & 1
\end{bmatrix}
$,
$ B= \begin{bmatrix}
1& 1 \\
0 & 1
\end{bmatrix}
$,
$ Q= \begin{bmatrix}
2& 0 \\
0 & 4
\end{bmatrix}
$
and
$ R= \begin{bmatrix}
0.5 & 0 \\
0& 0.25
\end{bmatrix}
$. Also, $\boldsymbol{x}(0) =\begin{bmatrix} 
10 &
5
\end{bmatrix}^T$.
\end{example}
\begin{solution*}
\normalfont
The optimization problem to be solved is given as:
\begin{equation*}
\min\limits_{(\boldsymbol{x},\boldsymbol{u}) \in \mathbb{R}^2 \times \mathbb{R}^2} s(\boldsymbol{x},\boldsymbol{u},t)
\end{equation*}
where
\begin{align*}
s=\frac{\partial q}{\partial t}+ \frac{\partial q}{\partial \boldsymbol{x}}
\left(
\begin{bmatrix}
0 & 1 \\
1 & 1
\end{bmatrix}
x+
\begin{bmatrix}
1 & 1 \\
0 & 1
\end{bmatrix}
\boldsymbol{u} \right)+\left(
\boldsymbol{x}^T
\begin{bmatrix}
2& 0 \\
0 & 4
\end{bmatrix}
\boldsymbol{x} +\boldsymbol{u}^T
\begin{bmatrix}
0.5 & 0 \\
0 & 0.25
\end{bmatrix}
\boldsymbol{u}
\right)
\end{align*} 
This function is non-convex and nonlinear. However, as shall be demonstrated the function can be convexified if $q(\boldsymbol{x},t)$ is chosen as per Proposition \ref{prop_1}.\\
Let the Krotov function be 
\begin{equation*}
q(x)=\boldsymbol{x}^TP\boldsymbol{x}, P= \begin{bmatrix}
p_{11}& p_{12}\\ p_{21} & p_{22}
\end{bmatrix}
\end{equation*}
Then the function $s(\boldsymbol{x},\boldsymbol{u},t)$ is given as:
\begin{align*}
s(\boldsymbol{x},\boldsymbol{u},t)&= \boldsymbol{x}^T (P+P^T) \left\{
\begin{bmatrix}
0 & 1 \\
1 & 1
\end{bmatrix}
\boldsymbol{x}+
\begin{bmatrix}
1 & 1 \\
0 & 1
\end{bmatrix}
\boldsymbol{u} \right\}
+ 2 x_1^{2}+4 x_2^{2}+ 0.5 u_1^{2}+0.25 u_2^{2} \\
&=x_1^2 \left(p_{12}+p_{21}+2 \right)+x_2^2 \left( p_{12}+p_{21}+2p_{22}+4\right)+ x_1x_2\left(2p_{11}+p_{12}+p_{21}+2p_{22} \right) +2u_1x_1 \left( p_{11}\right)\\
&\hspace{1 cm} +u_2 x_1 \left(2p_{11}+p_{12}+p_{21} \right)+ u_1x_2 \left( p_{12} +p_{21}\right)+ u_2 x_2 \left( p_{12}+p_{21}+2p_{22}\right)+0.5 u_1^{2}+0.25 u_2^{2}
\end{align*}
Next, using Proposition \ref{prop_1} and Corollary \ref{cor_1}, the direct solution can be obtained by choosing $P$ so as to satisfy 
\begin{equation} \label{ex_3_p_eqn}
PA+A^TP-\frac{1}{2} PBR^{-1}B^{T}P -\frac{1}{4} PBR^{-1}B^{T}P^T - \frac{1}{4} P^TBR^{-1}B^{T}P = 0
\end{equation}

Equation \eqref{ex_3_p_eqn} admits the four solutions:
$P_1= \begin{bmatrix} 
9.4172 & 6.0671 \\
6.8083 & -15.095
\end{bmatrix} $,
$P_2=
\begin{bmatrix}
-5.7559 & 12.756 \\
11.318 & -6.5319
\end{bmatrix}
$,
$P_3=
\begin{bmatrix}
5.7575 & -10.427 \\
-11.654 & 5.0354
\end{bmatrix}
$
and
$P_4=
\begin{bmatrix}
8.6906 & -5.9759 \\
-5.2647& 15.05
\end{bmatrix}.
$

Finally, to ensure the stability of the closed loop, $P$ is chosen such that Lemma \ref{lm_1} is satisfied i.e. $(P+P^T) \succ 0$. It can be easily verified that only $P_4$ satisifies this requirement. Thus, $P_4$ is the required value of matrix $P$ in this case.
Finally, the optimal control law is given as:
\begin{align*}
\boldsymbol{u}^{*}&= - \frac{1}{2}R^{-1} B^{T} (P+P^T) \boldsymbol{x}^{* }\\
&= \begin{bmatrix}
-1.2887 x_1^{*}+ 0.4267 x_2^{*} \\
-1.7240 x_1^{*}+ -5.0787x_2^{*}
\end{bmatrix}
\end{align*}
\end{solution*}
The closed loop response is shown in Figure \ref{res_inf_mimo_reg}.

\begin{figure}
\centering
\begin{minipage}{.5\textwidth}
\centering
\includegraphics[width=1 \linewidth]{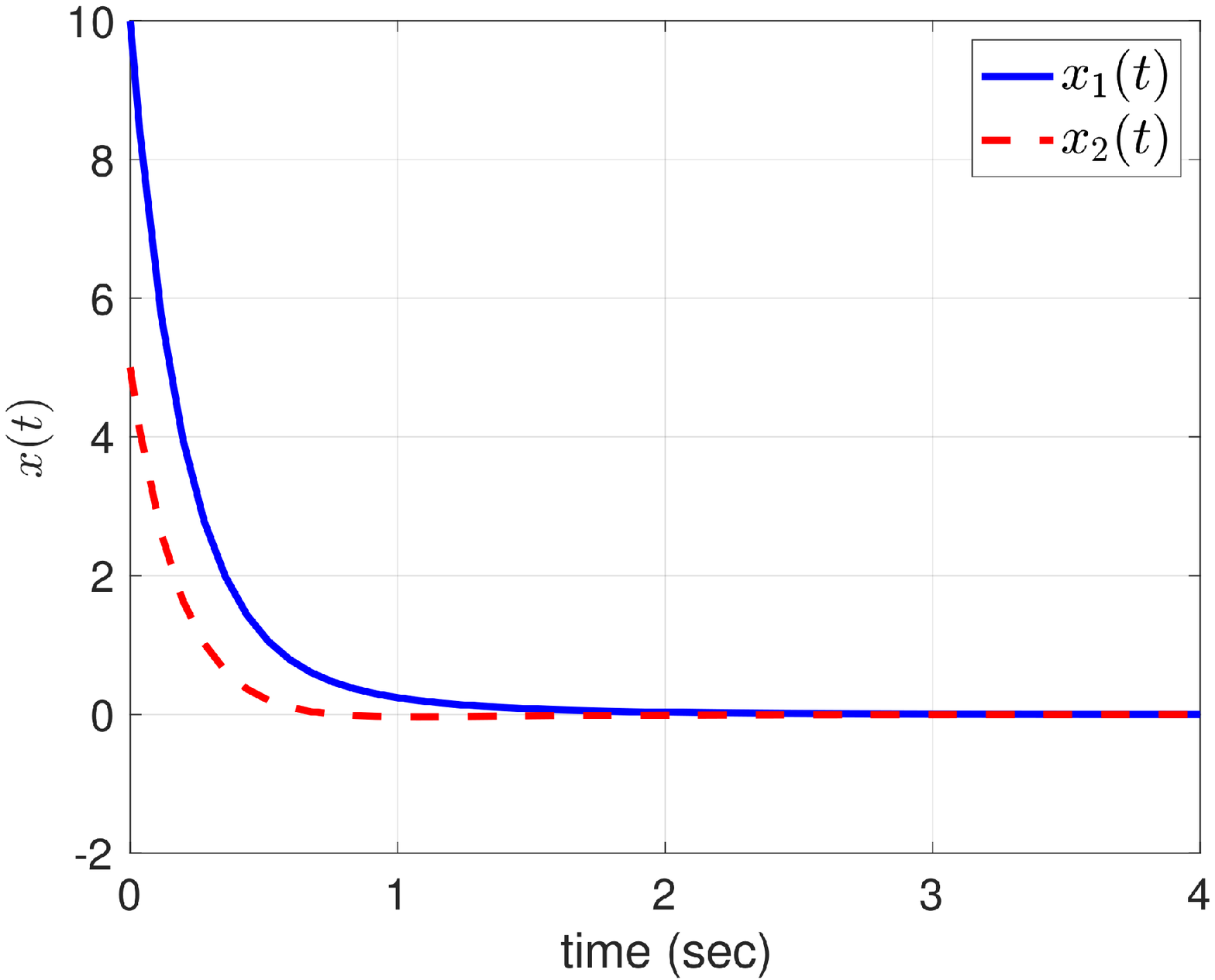}
\subcaption{Example \ref{ex5}-Regulation}
\label{res_inf_mimo_reg}
\end{minipage}%
\begin{minipage}{.5\textwidth}
\centering
\includegraphics[width=1\linewidth]{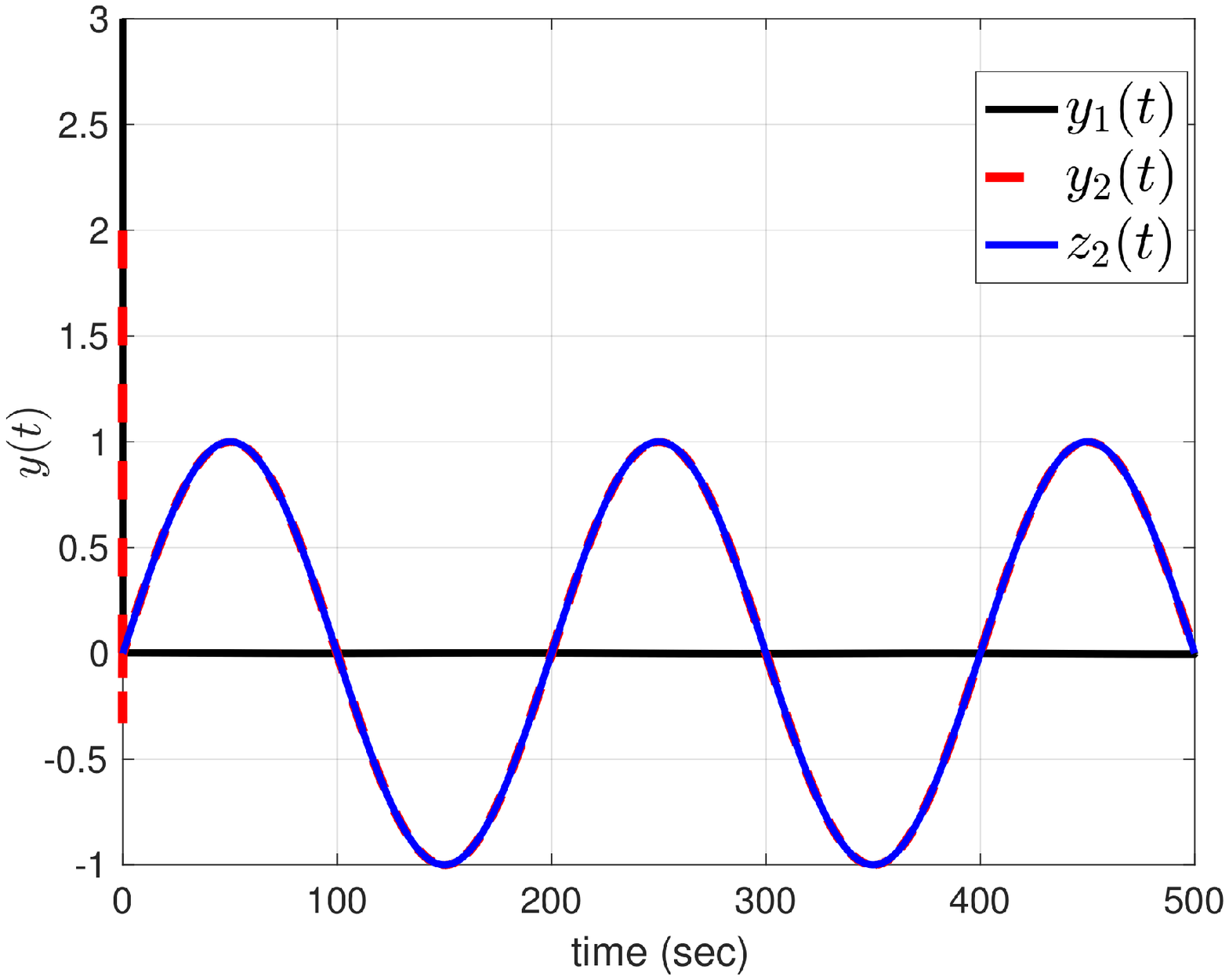}
\subcaption{Example \ref{ex6}-Tracking}
\label{res_inf_mimo_trac}
\end{minipage}
\caption{Closed loop response for Example \ref{ex5} and Example \ref{ex6}}
\label{closlooprestrac}
\end{figure}

\begin{example} \label{ex6}
For the MIMO system:
$\dot{\boldsymbol{x}}= A \boldsymbol{x}+ B\boldsymbol{u} $; $\boldsymbol{y}=C\boldsymbol{x}$
compute an optimal control law to minimize the performance index:
$
J= 0.5\int_{0}^{\infty} \boldsymbol{e}^T Q \boldsymbol{e} + \boldsymbol{u}^T R \boldsymbol{u} \ dt
$
where
$ A= \begin{bmatrix}
0 & 1 \\
1 & 1
\end{bmatrix}
$,
$ B= \begin{bmatrix}
1& 1 \\
0 & 1
\end{bmatrix}
$,
$ C= \begin{bmatrix}
1& 0 \\
0 & 1
\end{bmatrix}
$
$ Q= \begin{bmatrix}
200& 0 \\
0 & 400
\end{bmatrix}
$
and
$ R= \begin{bmatrix}
0.5 & 0 \\
0& 0.25
\end{bmatrix}
$
where $\boldsymbol{e}$ is the error defined as $\boldsymbol{e}=\boldsymbol{y}-\boldsymbol{z}$ and $\boldsymbol{z}$ is the reference defined as $\boldsymbol{z}= \begin{bmatrix} 0 \\ sin(\omega t) \end{bmatrix}$ with $\omega =0.01\pi$. Also, $\boldsymbol{x}(0)=
\begin{bmatrix}
5 & 2
\end{bmatrix}^T$.
\end{example}
\begin{solution*}
\normalfont
The optimization problem to be solved is given as:
\begin{equation*}
\min\limits_{(\boldsymbol{x},\boldsymbol{u}) \in \mathbb{R}^2 \times \mathbb{R}^2} s(\boldsymbol{x},\boldsymbol{u},t)
\end{equation*}
where
\begin{align*}
s=\frac{\partial q}{\partial t}+ \frac{\partial q}{\partial \boldsymbol{x}}
\left\{
\begin{bmatrix}
0 & 1 \\
1 & 1
\end{bmatrix}
\boldsymbol{x}+
\begin{bmatrix}
1 & 1 \\
0 & 1
\end{bmatrix}
\boldsymbol{u} \right\}+\left\{
\boldsymbol{e}^T
\begin{bmatrix}
200& 0 \\
0 & 400
\end{bmatrix}
\boldsymbol{e} +\boldsymbol{u}^T
\begin{bmatrix}
0.5 & 1 \\
0 & 0.25
\end{bmatrix}
\boldsymbol{u}
\right\}
\end{align*} 
We use the Krotov function as 
\begin{equation*}
q(\boldsymbol{x},t)= \boldsymbol{x}^TP\boldsymbol{x}-2\boldsymbol{g}^T\boldsymbol{x}; P= \begin{bmatrix} p_{11} &p_{12} \\ p_{21} & p_{22} \end{bmatrix};\boldsymbol{g}(t)= \begin{bmatrix} g_1(t) \\g_2(t) \end{bmatrix}
\end{equation*}
Then, the function $s(\boldsymbol{x},\boldsymbol{u},t)$ is given by:
\begin{align*}
s=&x_1^2 \left(p_{12}+p_{21}+200 \right)+x_2^2 \left( p_{12}+p_{21}+2p_{22}+400 \right)+ x_1x_2\left(2p_{11}+p_{12}+p_{21}+2p_{22} \right) +2u_1x_1 \left( p_{11}\right)\\
&\hspace{0.5 cm} +u_2 x_1 \left(2p_{11}+p_{12}+p_{21} \right)+ u_1x_2 \left( p_{12} +p_{21}\right)+ u_2 x_2 \left( p_{12}+p_{21}+2p_{22} \right) -2g_1x_2-2g_1u_1-2g_1u_2\\
&\hspace{0.7 cm}-2g_2 x_1 -2g_2 x_2 -2 g_2 u_2+400 sin^2(\omega t) 
\end{align*}
Again, it is easily verified that the function $s(\boldsymbol{x},\boldsymbol{u},t)$ is indeed convex for the selection of $P$ and $\boldsymbol{g}(t)$ as in Corollary \ref{cor_2}. $P$ is computed using \eqref{ex_3_p_eqn} similar to the case of regulation and {$\boldsymbol{g}(t)$ is a time varying function which can be computed as the steady state solution as in \eqref{g_lqt_eqn}
\begin{equation*}
\boldsymbol{g}(t)=- \int_{t}^{\infty} exp^{\left[A^{T}-\frac{1}{2}(P+P^T)BR^{-1}B^{T} (\tau -t)\right]} C^{T} Q \boldsymbol{z}(\tau) d \tau
\end{equation*}}
$P$ matrix is choosen to be (similar to the case of regulation) i.e.
$P=
P_4
$
and $\boldsymbol{g}(t)$ is calculated to be:
\begin{equation*}
\boldsymbol{g}(t)=\begin{bmatrix} -2.857sin(0.0314t) + 0.003cos(0.0314t) \\ 0.0039cos(0.031t) - 5.68sin(0.031t)\end{bmatrix}
\end{equation*}

The plots of $g_1(t)$ and $g_2(t)$ are shown in Figure \ref{gplotstrac} which clearly show their boundedness. Finally, the optimal control law is calculated using Corollary \ref{cor_2} and the closed loop response is given in Figure \ref{res_inf_mimo_trac}.

\begin{figure}
\centering
\centering
\includegraphics[width=1 \linewidth]{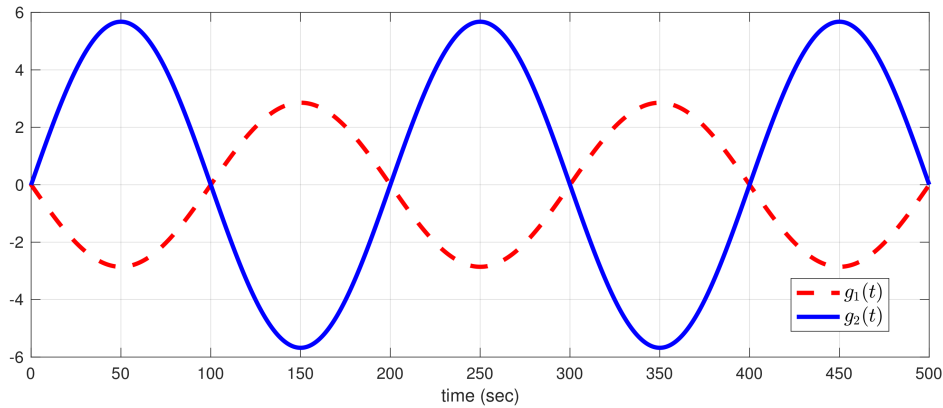}
\caption{Plots of $g_1(t)$ and $g_2(t)$ for Example \ref{ex6}}
\label{gplotstrac}
\end{figure}

\end{solution*}
\section{Conclusion} \label{conc}
In this paper, we propose a novel method to compute globally optimal control law for linear quadratic regulation and tracking problems based on Krotov sufficient conditions. The solution to the linear optimal control problem has been widely addressed in the literature using the celebrated CoV/HJB methods. These methods synthesize the global optimal control law, which is unique and requires some forced assumptions. In order to address this issue, we solved the optimal control problem using Krotov sufficient conditions, which does not require the require the notion of co-states (and hence the related assumptions), and the existence of the continuously differentiable optimal cost function.The idea behind Krotov formulation is that the original optimal control problem is translated into another equivalent optimization problem utilizing the so-called extension principle. The resulting optimization problem is highly nonlinear and non-convex, which is generally solved using iterative methods \cite{HalAgr17,Raf18} to yield the globally optimal solution. The angle of our attack is to compute a non-iterative solution, which is achieved by imposing convexity conditions on the equivalent optimization problem. As a byproduct, the selection of Krotov function now becomes very crucial, which shall be addressed in the future specifically for nonlinear and constrained optimal control problems. Finally, this work may serve as  background for further exploration and exploitation of the Krotov conditions for addressing more complex optimal control problems viz. non linear and distributed optimal control problems.

\bibliography{kro_ref}
\bibliographystyle{ieeetr}
\end{document}